\documentclass[a4paper,12pt]{amsart}
\usepackage{fullpage}
\usepackage[utf8]{inputenc}
\usepackage{array}
\usepackage{cite}
\usepackage{amssymb,amsthm,amsmath,mathrsfs,amsfonts}
\usepackage{fancyhdr}
\usepackage{amsmath,amssymb,amsthm}
\usepackage{amssymb, amsmath,amsthm,xcolor,bm,hyperref,color,tikz,euscript}
\usepackage{xcolor,bm,hyperref,tikz,euscript,sansmath,ytableau, mathtools}
\usetikzlibrary{3d,calc}
\usetikzlibrary{positioning,matrix,arrows,decorations.pathmorphing}
\usepackage{amssymb,amsthm,amsmath,mathrsfs}
\usepackage{graphicx,enumerate, microtype}
\usepackage{xcolor,bm,color,tikz,euscript,sansmath,ytableau}
\usepackage{tikz-cd}
\usepackage{subfig}
\usepackage{multirow}

\usepackage[all]{xy}

\usepackage{hyperref}
\usepackage{url, hypcap}
\definecolor{darkblue}{RGB}{0,0,160}
\hypersetup{
colorlinks,%
citecolor=darkblue,%
filecolor=black,%
linkcolor=darkblue,%
urlcolor=darkblue
}

\tikzstyle{invisivertex} = [black, shape=rectangle, minimum size=0pt, inner sep=3pt]
\tikzstyle{circledvertex} = [violet!80, draw, shape=circle, minimum size=0pt, inner sep=3pt]
\tikzstyle{dot} = [black, fill, shape=circle, minimum size=3pt, inner sep=2pt]
\tikzstyle{opendot} = [draw, black, shape=circle, minimum size=6pt, inner sep=1pt]
\tikzstyle{edgelabel} = [magenta, shape=circle, minimum size=6pt, inner sep=1pt]

\usepackage[colorinlistoftodos]{todonotes}
\usepackage{cleveref}
\usepackage[all]{xy}

\usepackage[utf8]{inputenc}
\usepackage[english]{babel}
\usepackage[utf8]{inputenc}
\usepackage[english]{babel}
\usepackage{bbm}

\newcommand{\ie}{{\textit{i.e.}}}

\newcommand{\cA}{\mathcal{A}}

\newcommand{\cL}{\mathcal{L}}

\newcommand{\Z}{\mathbb{Z}}
\newcommand{\QQ}{\mathbb{Q}}
\newcommand{\cC}{\mathcal{C}}
\newcommand{\cD}{\mathcal{D}}
\newcommand{\cM}{\mathcal{M}}
\newcommand{\idInt}{\mathsf{IL}^{\ast}}
\newcommand{\Int}{\mathsf{IL}}

\newcommand{\aaa}{\mathbf{a}}
\newcommand{\bbb}{\mathbf{b}}
\newcommand{\ccc}{\mathbf{c}}
\newcommand{\ddd}{\mathbf{d}}

\newcommand{\Poin}{\mathcal{\sf Poin}}
\newcommand{\POIN}[3][]{\Poin_{#1}(#2; #3)} 

\newcommand{\peak}{{\sf Peak}}
\newcommand{\Des}{{\sf Des}}
\newcommand{\LinExt}{{\sf LinExt}}

\newcommand{\supp}{{\sf supp}}
\newcommand{\StatMon}{{\sf u}}

\newcommand{\mon}{{\sf m}}
\newcommand{\incdec}{{\sf Alt}}
\newcommand{\PIE}{{\sf A}}
\newcommand{\topPIE}{{\sf B}}

\newcommand{\rank}{{\sf rank}}
\newcommand{\irank}{{\sf irank}}
\newcommand{\Rank}{{\sf Rank}}
\newcommand{\IRank}{{\sf IRank}}
\newcommand{\wt}{{\sf wt}}
\newcommand{\wtn}{\wt^+}
\newcommand{\wtp}{\wt^-}
\newcommand{\N}{\mathbb{N}}
\newcommand{\Num}{{\sf Num}}
\newcommand{\TT}[2]{{\sf T}(#1, #2)} 
\newcommand{\bout}{\operatorname{b-out}}

\newcommand{\NN}{\mathbb N}
\newcommand{\des}{\mathsf{Des}}
\newcommand{\m}{\mathsf{m}}

\newcommand{\multichain}[1]{\{\!\!\{#1\}\!\!\}}

\newcommand{\extPsi}{{\text{\scriptsize\sf ex}\hspace*{-1pt}}\Psi}
\newcommand{\pullPsi}{\Psi_{\sf pull}}
\newcommand{\combextPsi}{{\text{\scriptsize\sf cx}\hspace*{-1pt}}\Psi}
\newcommand{\combpullPsi}{\Psi_{\sf cpull}}
\newcommand{\combPsi}{\Psi_{\sf c}}

\newcommand{\QSym}{\textsf{QSym}}

\newcommand{\Dfn}[1]{\emph{\bfseries #1}}
\newtheorem{lemma}{Lemma}

\newtheorem{proposition}[lemma]{Proposition}
\newtheorem{thm}[lemma]{Theorem}
\newtheorem{corollary}[lemma]{Corollary}
\newtheorem{conjecture}[lemma]{Conjecture}

\theoremstyle{definition}
\newtheorem{definition}[lemma]{Definition}
\newtheorem{to do}[lemma]{To Do}

\newtheorem{remark}[lemma]{Remark}

\newtheorem{example}[lemma]{Example}

\numberwithin{lemma}{section}

\title{The Poincaré-extended $\aaa\bbb$-index}
\date{\today}

\author[G.~Dorpalen-Barry]{Galen Dorpalen-Barry$^\star$}
\address[G.~Dorpalen-Barry]{Texas A\&M University, USA}
\email{dorpalen-barry@tamu.edu}

\author[J.~Maglione]{Joshua Maglione$^\dagger$}
\address[J.~Maglione]{University of Galway, Ireland}
\email{joshua.maglione@universityofgalway.ie}

\author[C.~Stump]{Christian Stump $^\ddagger$}
\address[C.~Stump]{Ruhr-Universit\"at Bochum, Germany}
\email{christian.stump@rub.de}

\makeatletter
\let\@wraptoccontribs\wraptoccontribs
\makeatother

\contrib[with an appendix by]{Ricky Ini Liu}


\begin{document}

\begin{abstract}
  Motivated by a conjecture concerning Igusa local zeta functions for intersection posets of hyperplane arrangements, we introduce and study the \emph{Poincaré-extended $\aaa\bbb$-index}, which generalizes both the $\aaa\bbb$-index and the
  Poincaré polynomial.
  For posets admitting $R$-labelings, we give a
  combinatorial description of the coefficients of the extended
  $\aaa\bbb$-index, proving their nonnegativity.
  In the case of intersection posets of hyperplane arrangements,
  we prove the above conjecture of the second author and Voll as well as another conjecture of the second
  author and Kühne.
  We also define the pullback $\aaa\bbb$-index, generalizing the $\ccc\ddd$-index of face posets for oriented matroids.
  Our results recover, generalize and unify results from Billera--Ehrenborg--Readdy, Bergeron--Mykytiuk--Sottile--van~Willigenburg, Saliola--Thomas, and Ehrenborg.
  This connection allows us to translate our results into the language of quasisymmetric functions, and---in the special case of symmetric functions---pose a conjecture about Schur positivity.
  This conjecture was strengthened and proved by Ricky Liu, and the proof appears as an appendix.
\end{abstract}

\thanks{
  Supported by DFG Heisenberg grant STU 563/\{4-6\}-1 ``Noncrossing phenomena in Algebra and Geometry'' ($\star,\ddagger$) and by DFG-GRK 2297 ``Mathematical Complexity Reduction'' ($\dagger$).
}

\maketitle

\section{Introduction}

Grunewald, Segal, and Smith introduced the subgroup zeta function of finitely-generated groups~\cite{Grunewald-Segal-Smith}, and
Du Sautoy and Grunewald gave a general method to compute such zeta functions using $p$-adic integration and resolution of singularities~\cite{duSautoy-Grunewald}.
This motivated Voll and the second author to consider the setting where the multivariate polynomials factor linearly and found that the $p$-adic integrals are specializations of multivariate rational functions depending only on the combinatorics of the corresponding hyperplane arrangement~\cite{Maglione-Voll}.
After a natural specialization, its denominator greatly simplifies, and they conjecture that the numerator polynomial has nonnegative coefficients.

\medskip
In this work, we prove their conjecture, which is related to the poles of these zeta functions; see \Cref{rem:implications-for-zetas}.
Specifically, we reinterpret these numerator polynomials by introducing and studying the \emph{(Poincaré-)extended $\aaa\bbb$-index}, a polynomial generalizing both the \emph{Poincaré polynomial} and \emph{$\aaa\bbb$-index} of the \emph{intersection poset} of the arrangement.
These polynomials have been studied extensively in combinatorics, although from different perspectives.
The coefficients of the Poincaré polynomial have interpretations in terms of the combinatorics and the topology of the
arrangement~\cite[Section~2.5]{BjornerEtAl}.
The $\aaa\bbb$-index, on the other hand, carries information about the order complex of the poset and is
particularly well-understood in the case of face posets of oriented
matroids---or, more generally, Eulerian posets. In those settings, the
$\aaa\bbb$-index encodes topological data via the \emph{flag
$f$-vector}~\cite{bayer-survey}.

\medskip
We study the extended $\aaa\bbb$-index in the generality of graded posets admitting $R$-labelings.
This class of posets includes intersection posets of hyperplane arrangements and, more generally, geometric lattices and geometric semilattices.
We show that the extended $\aaa\bbb$-index has nonnegative coefficients by interpreting them in terms of a combinatorial statistic.
This generalizes statistics given for the $\aaa\bbb$-index by Billera, Ehrenborg, and Readdy~\cite{billera-ehrenborg-readdy} and for the pullback $\aaa\bbb$-index (defined below) by Bergeron, Mykytiuk, Sottile and van Willigenburg~\cite{bergeron-mykytiuk-sottile-vanWilligenburg2000}.
This interpretation proves the aforementioned conjecture~\cite{Maglione-Voll}, as well as a related conjecture from Kühne and the second author~\cite{Kuhne-Maglione}.

\medskip
We also describe a close relationship between the Poincaré polynomial and the $\aaa\bbb$-index by showing that the extended $\aaa\bbb$-index can be obtained from the $\aaa\bbb$-index by a suitable substitution.
This recovers, generalizes and unifies several results in the literature.
Concretely, special cases of this substitution were observed by Billera, Ehrenborg and Ready
for lattices of flats of \emph{oriented matroids}~\cite{billera-ehrenborg-readdy},
by Saliola and Thomas for lattices of flats of \emph{oriented interval greedoids}~\cite{saliola-thomas}, and by Ehrenborg for \emph{distributive lattices}~\cite{Ehrenborg-signed}.
In agreement with those special cases, our substitution allows us to study the \emph{pullback $\aaa\bbb$-index}, a polynomial in noncommuting variables $\ccc = \aaa + \bbb$ and $2\ddd = 2(\aaa\bbb + \bbb\aaa)$ with connections to the combinatorics of \emph{$P$-partitions} and \emph{quasisymmetric functions} from~\cite{stembridge-p-partitions} and~\cite[Section~7]{bergeron-mykytiuk-sottile-vanWilligenburg2000}.
We conjecture that our substitution (now defined on quasisymmetric functions) restricts to a map on symmetric function, and that the image of a Schur function is Schur positive.

\medskip

The remainder of this paper is organized as follows.
In \Cref{sec:main-stuff}, we introduce the main definitions (\Cref{sec:main-defs}) and state the main results (\Cref{sec:mainresults}) of the paper.
We then discuss in \Cref{sec:connections} how these recover, generalize and unify results in the context of $P$-partitions and quasisymmetric functions.
The proofs are then presented in \Cref{sec:proof-of-combinatorial-interp,,sec:refinement-of-ber}.

\subsection*{Acknowledgements}

This work was initiated at the \emph{Combinatorial Coworkspace} in March 2022 and continued at the conference \emph{Geometry Meets Combinatorics} the following September in Bielefeld.
We thank the organizers of both events for bringing us together and for fostering collaborative work environments.
We also thank
Aram Dermenjian, 
Martina Juhnke, and 
Vic Reiner for 
useful discussions and Richard Ehrenborg for drawing our attention to~\cite{Ehrenborg-signed}.
We finally thank Darij Grinberg for pointing us to the extended abstract ``The algebra of extended peaks''~\cite{2023arXiv230100309G}.
We believe that they also arrived at the same polynomials in the case of distributive lattices.

\section{The Poincaré-Extended $\aaa\bbb$-index}
\label{sec:main-stuff}

\subsection{Main definitions}
\label{sec:main-defs}

Unless otherwise specified,~$P$ is a finite \Dfn{graded poset} of rank~$n$.
That is, $P$ is a finite poset with unique minimum element~$\hat 0$ of rank~$0$ and unique maximum element~$\hat 1$ of rank~$n$ such that $\rank(X)$ is equal to the length of any maximal chain from~$\hat 0$ to $X$.
The \Dfn{M\"obius function}~$\mu$ of $P$ is given by $\mu(X, X) = 1$ for all $X\in P$ and $\mu(X, Y) = -\sum_{X\leq Z < Y} \mu(X, Z)$ for all $X < Y$ in $P$.
The \Dfn{Poincaré polynomial} of~$P$ is
\[
  \POIN{P}{y} = \sum_{X\in P} \mu(\hat{0}, X) \cdot (-y)^{\rank(X)} \ \in \Z[y].
\]
The \Dfn{chain Poincaré polynomial} of a chain $\cC = \big\{\cC_1 < \dots < \cC_k \big\}$ in~$P$ is 
\[
  \POIN[\cC]{P}{y} =\prod_{i = 1}^k \POIN{[\cC_i,\cC_{i+1}]}{y} \ \in \Z[y], 
\]
where we set~$\cC_{k+1} = \hat{1}$. By taking the singleton chain $\{\hat 0\}$,
we recover the usual Poincaré polynomial, $\POIN{P}{y} = \POIN[\{\hat 0\}]{P}{y}$. The set of ranks of a given chain~$\cC$ is given by 
\[
  \Rank(\cC) = \left\{ \rank(\cC_i) ~\middle|~ 1 \leq i \leq k \right\}. 
\]
We often consider polynomials in noncommuting variables~$\aaa$ and~$\bbb$ with coefficients being polynomials in $\Z[y]$.
For a subset $S \subseteq \{i,i+1,\dots,j\}$, we write $\mon_S = m_i\dots m_j$ for the monomial with $m_k = \bbb$ if $k \in S$ and $m_k = \aaa$ if $k \notin S$ and we similarly write $\wt_S = w_i\dots w_j$ for the polynomial with
\begin{equation}
\label{eqn:ext-weights}
  w_k = \begin{cases}
          \bbb        & \text{if } k\in S, \\
          \aaa - \bbb & \text{if } k\notin S\,.
        \end{cases}
\end{equation}
The supersets $\{i,i+1,\dots,j\}$ are understood from the context as the set of all indices that can possibly be contained in the set~$S$.
In case of ambiguity, we in addition identify the considered superset.
For a chain $\cC$ in~$P$, we also set $\mon_\cC = \mon_{\Rank(\cC)}$ and $\wt_\cC = \wt_{\Rank(\cC)}$.
The following is the main object of study of this paper.

\begin{definition}
\label{def:extab}
  The (\Dfn{Poincaré-})\Dfn{extended $\aaa\bbb$-index} of~$P$ is
  \[
    \extPsi(P;y,\aaa,\bbb)
      = \sum_{\cC\text{ chain in } P\setminus\{\hat{1}\}} \POIN[\cC]{P}{y} \cdot \wt_{\cC} \ \in \Z[y]\langle\aaa,\bbb\rangle\,,
  \]
  where~$\wt_\cC = w_0 \cdots w_{n-1}$ is given in \Cref{eqn:ext-weights}.
\end{definition}

Since~$P$ has a unique minimum, we always have $\POIN{P}{0} = 1$, implying
\[
  \extPsi(P;0,\aaa,\bbb)
    = \sum_{\cC\text{ chain in } P\setminus\{\hat{1}\}} \wt_\cC\,.
\]
This recovers\footnote{This is actually a mild variant of the usual definition of the $\aaa\bbb$-index; see~\Cref{rem:bottom-element-toss-out} below.} the \Dfn{$\aaa\bbb$-index}
\[
  \Psi(P;\aaa,\bbb) = \extPsi(P;0,\aaa,\bbb)\,.
\]

\begin{example}
\label{ex:face-poset2}
  We compute the extended $\aaa\bbb$-index of the poset $\cL$ drawn below on the left.
  On the right, we collect the relevant data.
  
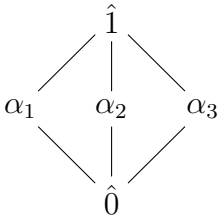
\begin{center}
    \begin{tikzpicture}[scale=1.2, baseline={([yshift=-.5ex]current bounding box.center)}]
      \node[invisivertex] (1)  at ( 0,2){$\hat 1$};
      \node[invisivertex] (H1) at (-1,1){$\alpha_1$};
      \node[invisivertex] (H2) at ( 0,1){$\alpha_2$};
      \node[invisivertex] (H3) at ( 1,1){$\alpha_3$};
      \node[invisivertex] (0)  at ( 0,0){$\hat 0$};
      \draw (0) -- (H1) -- (1);
      \draw (0) -- (H2) -- (1);
      \draw (0) -- (H3) -- (1);
    \end{tikzpicture}
     \hspace{2cm}
    \begin{tabular}{|c|c|c|c|}
      \hline
      &&&\\[-10pt]
      $\cC$ & $\POIN[\cC]{\cL}{y}$ & $\Rank(\cC)$ & $\wt_\cC$ \\[5pt]
      \hline
      &&&\\[-10pt]
      $\{\}$ & $1$ & $\{\}$ & $(\aaa-\bbb)^2$ \\[5pt]
      $\{\hat 0\}$ & $1 + 3y + 2y^2$ & $\{0\}$ & $\bbb(\aaa-\bbb)$ \\[5pt]
      $\{\alpha_i\}$ & $1+y$ & $\{1\}$ & $(\aaa-\bbb)\bbb$ \\[5pt]
      $\{\hat 0 < \alpha_i\}$ & $(1+y)^2$ & $\{0,1\}$ & $\bbb^2$\\
      \hline
    \end{tabular}
  \end{center}
  The extended $\aaa\bbb$-index and its specialization to the $\aaa\bbb$-index are thus
  \begin{align*}
    \extPsi(\cL;y,\aaa,\bbb)
      &= (\aaa-\bbb)^2\! +\! (1 + 3y + 2y^2)\bbb(\aaa-\bbb)\! +\! 3\cdot (1+y)(\aaa-\bbb)\bbb\! +\! 3\cdot(1+y)^2\bbb^2 \\
      &= \aaa^2 + (3y+2y^2)\bbb\aaa + (2+3y)\aaa\bbb + y^2\bbb^2, \\
    \Psi(\cL;\aaa,\bbb) &= \aaa^2 + 2\aaa\bbb\,.
  \end{align*}
\end{example}

\begin{remark}
  \label{rem:top-element-toss-out}
  Taking chains~$\cC$ in $P \setminus \{\hat 1\}$, rather than in~$P$, is a
  harmless reduction in the definition of the extended $\aaa\bbb$-index since
  $\POIN[\cC]{P}{y} = \POIN[\cC\cup\{\hat 1\}]{P}{y}$. If we consider both~$\cC$
  and~$\cC \cup\{\hat 1\}$ separately as summands of $\extPsi(P;y,\aaa,\bbb)$,
  we would need to consider weights $\wtn_\cC = w_0 \cdots w_n$
  taking also the $n$-th position into account.
  We would have the two terms $\POIN[\cC]{P}{y}\cdot\wtn_{\cC}$ and
  $\POIN[\cC \cup \{\hat 1\}]{P}{y}\cdot\wtn_{\cC \cup \{\hat 1\}}$,
  differing only in the last entry of the weight, so their sum is $\POIN[\cC]{P}{y}\cdot\wt_{\cC}\cdot \aaa$.
  This argument holds for all chains, proving
  \begin{equation}
    \extPsi(P;y,\aaa,\bbb)\cdot\aaa = \sum_{\cC\text{ chain in } P} \POIN[\cC]{P}{y} \cdot \wtn_\cC.
    \label{eq:enab01}
  \end{equation}
\end{remark}

\begin{remark}
\label{rem:bottom-element-toss-out}
  In the case of the $\aaa\bbb$-index, a similar argument to the one in
  \Cref{rem:top-element-toss-out} implies that we could further restrict to
  chains in $P\setminus\{\hat 0, \hat 1\}$. Therefore, we have 
  \begin{equation}
    \Psi(P;\aaa,\bbb) = \aaa\cdot\left(\sum_{\cC\text{ chain in } P\setminus\{\hat 0,\hat 1\}} \wtp_\cC\right),
    \label{eq:ab01}
  \end{equation}
  where $\wtp_\cC = w_1 \dots w_{n-1}$ as neither~$0$ nor~$n$ can appear
  in $\Rank(\cC)$. The expression in the parentheses recovers the usual
  definition of the $\aaa\bbb$-index as given, for example, in~\cite[Section~2]{bayer-survey}. Using additional information, the extended $\aaa\bbb$-index
  admits a similar property to that in \Cref{eq:ab01}, which is given in
  \Cref{cor:enab0}.
\end{remark}

The fact that $\hat 1$ is included in every chain in the computation of the
chain Poincaré polynomial is inspired by the setting of hyperplane arrangements.
A (central, real) \Dfn{hyperplane arrangement} $\mathcal{A}$ is a finite collection of
hyperplanes in~$\mathbb{R}^d$, all of which have a
common intersection.
The \Dfn{lattice of flats} $\mathcal{L}$ of $\mathcal{A}$
is the poset of subspaces of $\mathbb{R}^d$ obtained from
intersections of subsets of the hyperplanes, ordered by reverse inclusion.
The open, connected components of the complement $\mathbb{R}^d\setminus \mathcal{A}$ are called
(open) \Dfn{chambers}.
The set of (closed) \Dfn{faces} $\Sigma$ is the set of \emph{closures} of chambers of $\cA$, together with all possible intersections of closures of chambers (ignoring intersections which are empty).
This set comes equipped with a natural partial order by reverse inclusion, and for this reason we refer to $\Sigma$ as the \Dfn{face poset} of $\mathcal{A}$.
There is an order-preserving, rank-preserving
surjection $\supp : \Sigma \twoheadrightarrow \cL$ sending a face to its affine
span~\cite[Proposition~4.1.13]{BjornerEtAl}. This map extends to chains, and the
fiber sizes are given, for $\cC = \{ \cC_1 < \dots < \cC_k \}\subseteq \cL$, by
\begin{equation}
  \#\supp^{-1}(\cC) = \prod_{i=1}^{k} \POIN{[\cC_i,\cC_{i+1}]}{1} = \POIN[\cC]{P}{1}, 
  \label{eq:supppre}
\end{equation}
with $\cC_{k+1} = \hat 1$; see~\cite[Proposition~4.6.2]{BjornerEtAl}.
This is the key motivation for the next definition.
\begin{definition}
\label{def:pullbackab}
  The \Dfn{pullback $\aaa\bbb$-index} of~$P$ is
  \[
    \pullPsi(P;\aaa,\bbb) = \extPsi(P;1,\aaa,\bbb).
  \]
\end{definition}

Let~$\Sigma$ be the face poset and $\cL$ the lattice of flats of a real central
hyperplane arrangement. Since~$\Sigma$ may not have a unique minimum element, we
formally add a minimum element~$\hat{0}$ and let $\Sigma \cup \{\hat 0 \}$ be
the resulting poset. Now, \Cref{eq:supppre} relates the $\aaa\bbb$-index of the face poset and the pullback $\aaa\bbb$-index of the lattice of flats by
\begin{equation}
  \Psi(\Sigma \cup\{\hat 0\};\aaa,\bbb) =  \aaa\cdot\pullPsi(\cL;\aaa,\bbb)\,.
  \label{eq:pullbackab}
\end{equation}
Note that this corresponds to the evaluation of $\extPsi(\Sigma\cup\{\hat 0\};y,\aaa,\bbb)$ at $y=0$
to the evaluation of $\extPsi(\cL;y,\aaa,\bbb)$ at $y=1$.
\Cref{eq:supppre} and thus also \Cref{eq:pullbackab} hold indeed in the more general context of oriented
matroids.

\begin{example}
\label{ex:pullbackab}
  The pullback $\aaa\bbb$-index of the poset from \Cref{ex:face-poset2} is
  \[
    \pullPsi(\cL;\aaa,\bbb) = \extPsi(\cL;1,\aaa,\bbb) = \aaa^2 + 5\bbb\aaa + 5\aaa\bbb + \bbb^2\,.
  \]
  Consider the arrangement of three lines in the plane through a nonempty intersection as shown below on the left in a way that emphasizes its face structure.
  Its lattice of flats is the poset $\cL$ from \Cref{ex:face-poset2}.
  To the right, we draw its face poset~$\Sigma$ with~$\hat 0$ included.
  
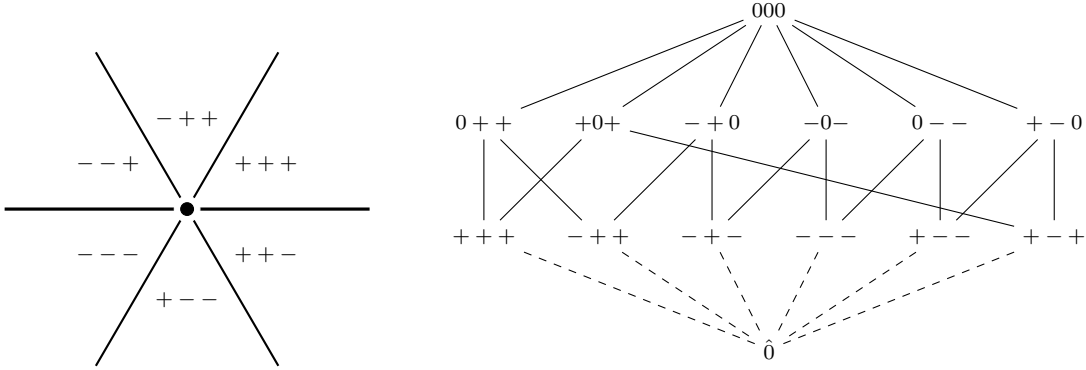
\begin{center}
     \begin{tikzpicture}[scale=1.2]
       \draw[very thick] (  0:-2) -- (  0:2);
       \draw[thick] (120:-2) -- (120:2);
       \draw[thick] (240:-2) -- (240:2);
       \draw[white, fill=white]  (0,0) circle [radius=4pt];
       \draw[fill]  (0,0) circle [radius=2pt];
       \node[invisivertex] at (30:1){\tiny$+++$};
       \node[invisivertex] at (90:1){\tiny$-++$};
       \node[invisivertex] at (150:1){\tiny$--+$};
       \node[invisivertex] at (210:1){\tiny$---$};
       \node[invisivertex] at (270:1){\tiny$+--$};
       \node[invisivertex] at (330:1){\tiny$++-$};
      \end{tikzpicture}
     \qquad
     \begin{tikzpicture}[scale=1.5]
      \node[invisivertex] (O) at (2.5,2){\tiny $000$};

      \node[invisivertex] (H1+) at (0,1){\tiny $0++$};
      \node[invisivertex] (H2+) at (1,1){\tiny $+0+$};
      \node[invisivertex] (H3-) at (2,1){\tiny $-+0$};
      \node[invisivertex] (H2-) at (3,1){\tiny $-0-$};
      \node[invisivertex] (H1-) at (4,1){\tiny $0--$};
      \node[invisivertex] (H3+) at (5,1){\tiny $+-0$};

      \node[invisivertex] (C1) at (0,0){\tiny $+++$};
      \node[invisivertex] (C2) at (1,0){\tiny $-++$};
      \node[invisivertex] (C3) at (2,0){\tiny $-+-$};
      \node[invisivertex] (C4) at (3,0){\tiny $---$};
      \node[invisivertex] (C5) at (4,0){\tiny $+--$};
      \node[invisivertex] (C6) at (5,0){\tiny $+-+$};

      \node[invisivertex] (1) at (2.5,-1){\tiny $\hat 0$};

      \path[-] (O) edge [] node[above] {} (H1+);
      \path[-] (O) edge [] node[above] {} (H2+);
      \path[-] (O) edge [] node[above] {} (H3-);
      \path[-] (O) edge [] node[above] {} (H2-);
      \path[-] (O) edge [] node[above] {} (H1-);
      \path[-] (O) edge [] node[above] {} (H3+);

      \path[-] (C1) edge [] node[above] {} (H1+);
      \path[-] (C1) edge [] node[above] {} (H2+);
      \path[-] (C2) edge [] node[above] {} (H1+);
      \path[-] (C2) edge [] node[above] {} (H3-);
      \path[-] (C3) edge [] node[above] {} (H3-);
      \path[-] (C3) edge [] node[above] {} (H2-);
      \path[-] (C4) edge [] node[above] {} (H2-);
      \path[-] (C4) edge [] node[above] {} (H1-);
      \path[-] (C5) edge [] node[above] {} (H1-);
      \path[-] (C5) edge [] node[above] {} (H3+);
      \path[-] (C6) edge [] node[above] {} (H3+);
      \path[-] (C6) edge [] node[above] {} (H2+);

      \path[-] (C1) edge [dashed] node[above] {} (1);
      \path[-] (C2) edge [dashed] node[above] {} (1);
      \path[-] (C3) edge [dashed] node[above] {} (1);
      \path[-] (C4) edge [dashed] node[above] {} (1);
      \path[-] (C5) edge [dashed] node[above] {} (1);
      \path[-] (C6) edge [dashed] node[above] {} (1);
      
      \end{tikzpicture}
  \end{center}
  The $\aaa\bbb$-index of $\Sigma\cup\{\hat 0\}$ can be computed as
  \[
    \aaa^3 + 5\aaa\bbb\aaa + 5 \aaa^2\bbb + \aaa\bbb^2 = \aaa(\aaa^2 + 5\bbb\aaa + 5\aaa\bbb + \bbb^2).
  \]
  As seen above, this equals $\aaa\cdot\pullPsi(\cL;\aaa,\bbb)$.
\end{example}

\subsection{Main results}
\label{sec:mainresults}

The main results of this paper concern \emph{$R$-labeled posets}.
These form a large family of posets including \emph{distributive lattices}, \emph{semimodular} (in particular \emph{geometric}) \emph{lattices}, and \emph{noncrossing partition lattices}.
In order to state \Cref{thm:combinatorial-interp}, we introduce a combinatorial statistic on maximal chains of these posets and use this to describe the extended $\aaa\bbb$-index.
In \Cref{sec:connections}, we briefly discuss this combinatorial statistic for general edge labeled graded posets.

\medskip
A function~$\lambda$ from the set of cover relations $X\lessdot Y$ in~$P$ into the positive integers is an \Dfn{$R$-labeling} of $P$ if, for every interval $[X, Y]$ in~$P$, there is a unique
maximal chain $X=\cM_i\lessdot \cM_{i+1} \lessdot \cdots \lessdot \cM_j=Y$ such that 
\[ 
  \lambda(\cM_i,\cM_{i+1}) \leq \lambda(\cM_{i+1},\cM_{i+2}) \leq \cdots \leq \lambda(\cM_{j-1},\cM_j). 
\] 
We say a poset $P$ is \Dfn{$R$-labeled} if it is finite, graded, and admits an
$R$-labeling.
Throughout this section, we consider $R$-labeled posets with a fixed $R$-labeling~$\lambda$.

\medskip
The first result is
a combinatorial statistic describing the coefficients of the extended
$\aaa\bbb$-index which witnesses their nonnegativity.
It generalizes~\cite[Corollary~7.2]{billera-ehrenborg-readdy} and also reproves it using purely combinatorial arguments.
For a maximal chain~$\cM =
\{\cM_0\lessdot\cM_1\lessdot\dots\lessdot\cM_n\}$ in $P$, define the monomial
$\StatMon(\cM) = u_1\cdots u_n$ in $\aaa,\bbb$ given by $u_1 = \aaa$
and for $i \in \{2,\dots,n\}$ by
\begin{equation}
\label{eq:statmon}
  u_i = \begin{cases}
                \aaa &\text{if } \lambda(\cM_{i-2},\cM_{i-1}) \leq \lambda(\cM_{i-1},\cM_{i})\,, \\
                \bbb &\text{if } \lambda(\cM_{i-2},\cM_{i-1}) > \lambda(\cM_{i-1},\cM_{i})\,.
              \end{cases}
\end{equation}
Now, let $E \subseteq \{1,\dots,n\}$, viewed as a subset of the cover relations
in the chain~$\cM$. Define the monomial $\StatMon(\cM,E) = v_1\dots v_n$ in
$\aaa,\bbb$ to be obtained from $\StatMon(\cM)$ by
\begin{itemize}
  \item replacing all variables~$\aaa$ by $\bbb$ at positions~$i\in \{1,\dots,
  n\}$ if $i \in E$ and
  \item replacing all variables~$\bbb$ by $\aaa$ at positions~$i\in \{2,\dots,
  n\}$ if $i-1 \in E$.
\end{itemize}
In symbols this means, for the given position~$i \in \{1,\dots,n\}$, that
\begin{equation*}
  \begin{array}{ll}
          v_i = \aaa &\text{ if } \begin{cases}
                              u_i = \aaa, \quad i\hspace*{21.5pt} \notin E\quad\text{or} \\
                              u_i = \bbb, \quad i-1 \in E\,,
                            \end{cases} \\ \\
          v_i = \bbb & \text{ if } \begin{cases}
                              u_i = \aaa, \quad i\hspace*{21.5pt} \in E\quad\text{or} \\
                              u_i = \bbb, \quad i-1 \notin E\,.
                            \end{cases}
  \end{array}
\end{equation*}
We have, in particular, $\StatMon(\cM,\emptyset) = \StatMon(\cM)$ and
\begin{equation}
\label{eq:firstindex}
  v_1 = \begin{cases}
          \aaa & \text{ if } 1 \notin E\,, \\
          \bbb & \text{ if } 1 \in E\,.
        \end{cases}
\end{equation}

\begin{thm}
\label{thm:combinatorial-interp}
  Let~$P$ be an $R$-labeled poset of rank~$n$.
 Then 
  \[
    \extPsi(P;y,\aaa,\bbb)
      = \sum_{(\cM,E)} y^{\#E}\cdot\StatMon(\cM,E)
  \]
  where the sum ranges over all maximal chains~$\cM$ in~$P$ and all subsets~$E \subseteq \{1,\dots,n\}$.
\end{thm}

When $P$ is a geometric lattice, setting $y = 0$ in \Cref{thm:combinatorial-interp} recovers~\cite[Corollary~7.2]{billera-ehrenborg-readdy}.
Specifically
\begin{equation}
  \Psi(P;\aaa,\bbb)
    = \sum_{\cM} \StatMon(\cM)\,,
\label{eq:abstat}
\end{equation}
where the sum ranges over all maximal chains~$\cM = \{ \cM_0 \lessdot \dots \lessdot \cM_n\}$.

\begin{example}
\label{ex:face-poset3}
  The poset from the previous examples admits the $R$-labeling given below on the left.
  On the right, we collect the relevant data to compute the combinatorial description of the extended $\aaa\bbb$-index.
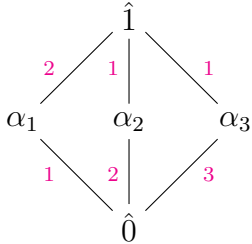
  \begin{center}
    \begin{tikzpicture}[scale=.7, baseline={([yshift=-.5ex]current bounding box.center)}]
      \node[invisivertex] (1) at (2,4){$\hat 1$};
      \node[invisivertex] (H1) at (0,2){$\alpha_1$};
      \node[invisivertex] (H2) at (2,2){$\alpha_2$};
      \node[invisivertex] (H3) at (4,2){$\alpha_3$};
      \node[invisivertex] (0) at (2,0){$\hat 0$};

      \node[edgelabel] at (.5,1){\tiny $1$};
      \node[edgelabel] at (1.7,1){\tiny $2$};
      \node[edgelabel] at (3.5,1){\tiny $3$};

      \node[edgelabel] at (.5,3){\tiny $2$};
      \node[edgelabel] at (1.7,3){\tiny $1$};
      \node[edgelabel] at (3.5,3){\tiny $1$};

      \draw (0) -- (H1) -- (1);
      \draw (0) -- (H2) -- (1);
      \draw (0) -- (H3) -- (1);
    \end{tikzpicture}
    \hspace*{2cm}
    \begin{tabular}{|c|c|c|c|c|}
    \hline
    & & & &\\[-10pt]
    $E$ & $y^{\#E}$
    & $\hat 0 \lessdot \alpha_1 \lessdot \hat 1$
    & $\hat 0 \lessdot \alpha_2 \lessdot \hat 1$ &
    $\hat 0 \lessdot \alpha_3 \lessdot \hat 1$ \\[5pt]
    \hline
    & & & &\\[-10pt]
    $\{\}$      & $1$   & $\aaa\aaa$ & $\aaa\bbb$ & $\aaa\bbb$ \\
    $\{ 1 \}$   & $y$   & $\bbb\aaa$ & $\bbb\aaa$ & $\bbb\aaa$ \\
    $\{ 2 \}$   & $y$   & $\aaa\bbb$ & $\aaa\bbb$ & $\aaa\bbb$ \\
    $\{ 1,2 \}$ & $y^2$ & $\bbb\bbb$ & $\bbb\aaa$ & $\bbb\aaa$\\[5pt]
    \hline
    \end{tabular}
  \end{center}
  Then $\extPsi(\cL;y,\aaa,\bbb)$ is obtained as
  \[
    \extPsi(\cL;y,\aaa,\bbb) = \aaa\aaa + (3y+2y^2)\bbb\aaa + (2 + 3y)\aaa\bbb + y^2\bbb\bbb\,,
  \]
  in agreement with our computation in \Cref{ex:face-poset2}.
\end{example}

\Cref{thm:combinatorial-interp} has many consequences, which we formulate in eight corollaries.
The most important gives a substitution sending the
$\aaa\bbb$-index to the extended $\aaa\bbb$-index---meaning that the extended
$\aaa\bbb$-index is already encoded in the $\aaa\bbb$-index.

\begin{corollary}
\label{thm:refinement-of-ber}
  For an $R$-labeled poset~$P$, we have
  \[
    \extPsi(P;y,\aaa,\bbb) = \omega\big(\Psi(P;\aaa,\bbb)\big)
  \]
  where the substitution $\omega$ replaces all occurrences of $\aaa\bbb$ with
  $\aaa\bbb + y\bbb\aaa + y\aaa\bbb + y^2\bbb\aaa$ and then simultaneously
  replaces all remaining occurrences of $\aaa$ with $\aaa+y\bbb$ and $\bbb$ with
  $\bbb+y\aaa$.
\end{corollary}

The proof of \Cref{thm:refinement-of-ber} relies on the fact that~$P$ is $R$-labeled.
We suspect however that this is true more generally.

\begin{conjecture}
  \label{conj:refinement-of-ber}
  For any finite and graded poset~$P$, we have
  \[
    \extPsi(P;y,\aaa,\bbb) = \omega\big(\Psi(P;\aaa,\bbb)\big)\,.
  \]
\end{conjecture}
We have tested this conjecture on all graded posets of cardinality at most~$10$ and also on many larger graded posets.

\begin{example}
  We have seen that the poset~$\cL$ from \Cref{ex:face-poset2} has $\aaa\bbb$-index
  \[
    \Psi(\cL;\aaa,\bbb) = \extPsi(\cL,0,\aaa,\bbb) = \aaa\aaa + 2\aaa\bbb\,.
  \]
  Applying $\omega$ gives
  \[
    \omega\big(\Psi(\cL;\aaa,\bbb)\big)
      = (\aaa+y\bbb)^2 + 2(\aaa\bbb + y\bbb\aaa + y\aaa\bbb + y^2\bbb\aaa) = \extPsi(\cL;y,\aaa,\bbb)\,,
  \]
  which coincides with the extended $\aaa\bbb$-index we computed in \Cref{ex:face-poset3}.
\end{example}

Using \Cref{thm:refinement-of-ber}, the monomials $\StatMon(\cM,E)$ in \Cref{thm:combinatorial-interp} capture the same information as the \emph{generalized descent sets} on \emph{réseaux} as defined by Bergeron, Mykytiuk, Sottile, and van Willigenburg in~\cite[Section~7]{bergeron-mykytiuk-sottile-vanWilligenburg2000} in the context of quasisymmetric functions.
The next corollary can be seen as a refinement of~\cite[Proposition~2.2]{stembridge-p-partitions} and of~\cite[Theorem~7.2]{bergeron-mykytiuk-sottile-vanWilligenburg2000}, stated in terms of $\aaa\bbb$-indices rather than quasisymmetric functions.
Both can be seen as the special case for the pullback $\aaa\bbb$-index: the first for \emph{enriched $P$-partitions} and the second for general edge-labeled graded posets, compare with \Cref{sec:connections}.
We start by describing their relevant combinatorics in the present notation.
Let $\cM$ be a maximal chain with $\StatMon(\cM) = u_1\dots u_n$, and let
\[
  \peak(\cM) = \big\{ i \in \{2,\dots,n\} \mid u_{i-1} = \aaa, u_i = \bbb \big\}
\]
denote its \Dfn{peak set}.
A set $S \subseteq \{1,\dots,n\}$ is then \Dfn{$\cM$-peak-covering} if
\[
  \peak(\cM) \subseteq S \cup \{ i+1 \mid i \in S\}\,,
\]
and let $\bout(\cM,S)$ be the number of positions $i \in \{1,\dots,n\} \setminus S$ for which $u_i = \bbb$.

\begin{corollary}
\label{cor:peakenumeration}
  For an~$R$-labeled poset $P$ of rank~$n$, we have
  \[
    \extPsi(P;y,\aaa,\bbb) = \sum_{(\cM,S)} (1+y)^{\# S}\cdot y^{\bout(\cM,S)} \cdot \wt_S\,,
  \]
  where the sum ranges over all maximal chains~$\cM$ and all $\cM$-peak-covering subsets $S \subseteq \{1,\dots,n\}$ and where $\wt_S = w_1\dots w_n$ as given in \Cref{eqn:ext-weights}.
\end{corollary}

\begin{example}
  For the poset~$\cL$ from \Cref{ex:face-poset2}, using \Cref{cor:peakenumeration} we compute
  \begin{align*}
    \extPsi(\cL;y,\aaa,\bbb)
    &= \underbrace{(\aaa-\bbb)^2}_{S = \emptyset} + \underbrace{(1+y) \cdot \bbb(\aaa-\bbb)}_{S = \{1\}} + \underbrace{(1+y) \cdot (\aaa-\bbb)\bbb}_{S = \{2\}} + \underbrace{(1+y)^2 \cdot\bbb^2}_{S = \{1,2\}} \\
    &+ 2\cdot\Big(\underbrace{(1+y) \cdot y \cdot \bbb(\aaa-\bbb)}_{S = \{1\}} + \underbrace{(1+y) \cdot (\aaa-\bbb)\bbb}_{S = \{2\}} + \underbrace{(1+y)^2 \cdot\bbb^2}_{S = \{1,2\}} \Big)
  \end{align*}
  where the sum in the first row corresponds to the maximal chain $\cM = \{ \hat 0 \lessdot \alpha_1 \lessdot \hat 1\}$ with $\StatMon(\cM) = \aaa\aaa$ and $\peak(\cM) = \emptyset$.
  The condition of being $\cM$-peak-covering is thus vacuous, and we sum over all $S \subseteq \{1,2\}$.
  The sum in the second row corresponds to the two maximal chains $\cM = \{ \hat 0 \lessdot \alpha_i \lessdot \hat 1\}$ for $i \in \{2,3\}$ with $\StatMon(\cM) = \aaa\bbb$ and $\peak(\cM) = \{2\}$.
  In this case, the condition of being $\cM$-peak-covering excludes the empty set, and we sum over all nonempty subsets $S \subseteq \{1,2\}$.
  Observe that only in the second row, we have the second letter of $\StatMon(\cM)$ equal to~$\bbb$, so the only set for which $\bout(\cM,S) \neq 0$ is $S = \{1\}$, and in this case $\bout(\cM,S) = 1$.
  The above sum further expands to
  \begin{align*}
    \extPsi(\cL;y,\aaa,\bbb)
    &= (\aaa\aaa + y\bbb\aaa + y\aaa\bbb + y^2\bbb\bbb)\ +\ 2\big((y+y^2)\bbb\aaa + (1+y)\aaa\bbb \big) \\
    &= \aaa\aaa + (3y+2y^2)\bbb\aaa + (2+3y)\aaa\bbb + \bbb\bbb
  \end{align*}
  as expected.
\end{example}

Another consequence of \Cref{thm:refinement-of-ber} is that the Poincaré polynomial
of~$P$ is in fact encoded in its $\aaa\bbb$-index. To see this, we define another
substitution~$\iota$, which deletes the first letter from every $\aaa\bbb$-monomial, so
$\iota(\aaa^3\bbb\aaa + (1+y)\bbb\aaa) = \aaa^2\bbb\aaa + (1+y)\aaa$ for example. 
This gives us a way to obtain the Poincaré polynomial from the $\aaa\bbb$-index, a result which is similar in spirit to~\cite[Proposition 5.3]{billera-ehrenborg-readdy}.

\begin{corollary}\label{eq:exabpoin}
  For an $R$-labeled poset $P$ of rank $n$, the Poincaré polynomial is the coeffcient of $\aaa^{n-1}$ in
  $\iota\big(\omega\big(\Psi(P;\aaa,\bbb)\big)\big)$.
\end{corollary}

\Cref{thm:refinement-of-ber}
generalizes~\cite[Theorem~3.1]{billera-ehrenborg-readdy} relating the
$\aaa\bbb$-index of the lattice of flats of an oriented matroid with the
$\aaa\bbb$-index of its face poset. As a consequence, we see that
$\extPsi(P;y,\aaa,\bbb)$ is akin to a refinement of a $\ccc\ddd$-index. We
make this observation precise in the following corollary. 

\begin{corollary}
\label{cor:cd-index}
  For an $R$-labeled poset~$P$, there exists a polynomial $\Phi(P;
  \ccc_1,\ccc_2,\ddd)$ in noncommuting variables $\ccc_1,\ccc_2,\ddd$ such that 
  \[ 
    \extPsi(P; y, \aaa, \bbb) = \Phi(P;\, \aaa+y\bbb,\, \bbb+y\aaa,\, \aaa \bbb + y \bbb \aaa + y \aaa \bbb + y^2 \bbb \aaa) .
  \]
  In particular, the pullback $\aaa\bbb$-index $\pullPsi(P;\aaa,\bbb)$ is a polynomial in noncommuting variables $\ccc = \aaa+\bbb$ and $2\ddd = 2(\aaa\bbb+\bbb\aaa)$.
\end{corollary}

\begin{remark}[The synthetic $\ccc\ddd$-index]\label{rem:matroids}
  Recall that the $\ccc\ddd$-index of a poset exists if the $\aaa\bbb$-index can be written as a polynomial in $\ccc = \aaa + \bbb$ and $\ddd = \aaa\bbb + \bbb\aaa$.
  Bayer, Fine, and Klapper observe that a poset satisfies the \emph{generalized Dehn-Sommerville relations} if and only if its $\ccc\ddd$-index exists and has nonnegative integer coefficients~\cite[Theorem 4]{bayer-klapper}.
  The $\ccc\ddd$-index of an Eulerian poset always exists (see~\cite[Theorem 2.1]{bayer-billera}) and has nonnegative coefficients when it comes from the face poset of a polytope (or, more generally, from a Gorenstein* poset)~\cite[Theorem 1.3]{karu}.
  
  In~\cite{billera-ehrenborg-readdy}, Billera, Ehrenborg, and Readdy give an elegant alternative proof of the nonnegativity of the $\ccc\ddd$-index of the face poset of an oriented matroid.
  They use the support map from \Cref{eq:supppre} to relate the $\aaa\bbb$-index of the lattice of flats to the $\aaa\bbb$-index of the face poset.
  In our language, they interpret (using posets and polytopes) the extended $\aaa\bbb$-index of an oriented matroid at $y=0$ and $y=1$.
  Every matroid admits an extended $\aaa\bbb$-index, and the evaluation at $y=0$ is the $\aaa\bbb$-index of its lattice of flats.
  This raises the natural question whether there is a geometric or poset-theoretic interpretation of the $y=1$ evaluation of the extended $\aaa\bbb$-index.
  For this reason, we call the $y=1$ evaluation of the extended $\aaa\bbb$-index rewritten in terms of $\ccc$ and $\ddd$ the \Dfn{synthetic $\ccc\ddd$-index}.
\end{remark}

  \begin{example}[The Fano matroid]
 The extended $\aaa\bbb$-index of the \emph{Fano matroid}~\cite[Example~6.6.2(1)]{BjornerEtAl} is 
  \begin{multline*}
    \aaa^3 + (7y+6)\aaa^2\bbb + (14y^2+21y+6)\aaa\bbb\aaa \\
    + (7y^2+14y+8)\aaa\bbb^2 + (8y^3+14y^2+7y)\bbb\aaa^2 \\
    + (6y^3+21y^2+14y)\bbb\aaa\bbb + (6y^3+7y^2)\bbb^2\aaa + y^3\bbb^3.
  \end{multline*}
  Setting $y=1$ and then $\ccc = \aaa + \bbb$ and $\ddd = \aaa\bbb + \bbb\aaa$ gives the synthetic $\ccc\ddd$-index $12\ccc\ddd + 28\ddd\ccc + \ccc^3$.
  A convex $3$-polytope with this $\ccc\ddd$-index would have $30$~vertices and $14$~facets; see~\cite{Meisinger}.
 Thus its polar dual polytope would have $14$~vertices and $30$~facets, contradicting the Upper Bound Theorem~\cite[Theorem 8.23]{ziegler}.
\end{example}

 \begin{example}[The Mac Lane matroid]
 We compute the synthetic $\ccc\ddd$-index of the \emph{Mac Lane matroid}; see~\cite[page~114]{bland-lasvergnas} and~\cite[Section~2]{ziegler-nonorientable}.
  We get the synthetic $\ccc\ddd$-index $18\ccc\ddd + 32\ddd\ccc + \ccc^3$, which is the $\ccc\ddd$-index of the polar dual of the convex hull of the following $20$ points 
  \begin{center}
    \begin{tabular}{ c c c c c }
    $\left(\frac{17}{18}, \frac{8}{9}, \frac{11}{12}\right)$,&
    $\left(\frac{17}{18}, \frac{5}{49}, \frac{5}{16}\right)$,&
    $\left(\frac{19}{51}, \frac{11}{15}, \frac{6}{43}\right)$,&
    $\left(\frac{17}{18}, \frac{4}{5}, \frac{1}{2}\right)$,&
    $\left(\frac{7}{32}, \frac{3}{4}, \frac{10}{13}\right)$,\\
    & & & & \\[-7pt]
    $\left(\frac{2}{3}, \frac{7}{9}, \frac{18}{19}\right)$,&
    $\left(\frac{1}{26}, \frac{19}{21}, \frac{11}{85}\right)$,&
    $\left(\frac{5}{6}, \frac{6}{7}, \frac{14}{15}\right)$,&
    $\left(\frac{3}{4}, \frac{4}{75}, \frac{2}{5}\right)$,&
    $\left(\frac{1}{9}, \frac{15}{16}, \frac{9}{19}\right)$,\\
    & & & & \\[-7pt]
    $\left(\frac{7}{8}, \frac{1}{5}, \frac{10}{47}\right)$,&
    $\left(\frac{21}{44}, \frac{10}{11}, \frac{18}{19}\right)$,&
    $\left(\frac{9}{10}, \frac{16}{17}, \frac{5}{6}\right)$,&
    $\left(\frac{3}{4}, \frac{16}{33}, \frac{13}{14}\right)$,&
    $\left(\frac{17}{18}, \frac{21}{43}, \frac{8}{9}\right)$,\\
    & & & & \\[-7pt]
    $\left(\frac{17}{18}, \frac{7}{8}, \frac{9}{11}\right)$,&
    $\left(\frac{3}{17}, \frac{11}{145}, \frac{19}{34}\right)$,&
    $\left(\frac{13}{28}, \frac{15}{16}, \frac{17}{31}\right)$,&
    $\left(\frac{9}{137}, \frac{13}{38}, \frac{3}{8}\right)$,&
    $\left(\frac{3}{8}, \frac{20}{69}, \frac{20}{21}\right)$.
  \end{tabular}
\end{center}
\end{example}

\begin{remark}[Oriented interval greedoids]
\label{rem:oig-version}
  The argument used for oriented matroids and their lattices of flats also applies to \emph{oriented interval greedoids}, where the analogue of \Cref{eq:supppre} is given in~\cite[Theorem~6.8]{saliola-thomas}.
  Since the lattice of flats of an interval greedoid is a semimodular lattice, it admits an $R$-labeling; see~\cite[Theorem~3.7]{bjorner-shellings}.
  Applying \Cref{thm:refinement-of-ber} and setting $y=1$ gives~\cite[Corollary~6.12]{saliola-thomas}.
\end{remark}

\begin{remark}[Distributive lattices \& $r$-signed Birkhoff posets]
\label{rem:jp-version}
  Ehrenborg discussed an $\omega$-like substitution for arbitrary distributive lattices~\cite{Ehrenborg-signed}.
  Remarkably, that substitution is equivalent to the substitution in \Cref{thm:refinement-of-ber} for $y = r-1 \in \N$.
  In that case of distributive lattices, the parameter~$r$ is a fixed integer (rather than a variable) carrying information about the fiber sizes of a certain support map.
  For a (not necessarily graded) finite poset~$P$, the \emph{$r$-signed Birkhoff poset} $J_r(P)$ is the collection of pairs $(F,f)$ where~$F$ is an \emph{order ideal} in~$P$ and~$f$ is a map from the maximal elements in~$F$ to the set $\{1,\dots,r\}$, with order relation given by
  \[
    (F,f) \leq (G,g) \quad \Longleftrightarrow \quad G\subseteq F \text{ and } f(x) = g(x) \text{ for all }x \in \max(F) \cap \max(G)\,.
  \]
  These posets were defined in~\cite{Hsiao,Ehrenborg-signed} and studied in connection to the Birkhoff lattice $J(P) = J_1(P)$.
  The map $z: J_r(P) \to J(P)$ with $(F,f) \mapsto F$ is an order- and rank-preserving poset surjection for which the fiber size of a chain~$\cC$ in $J(P)$ can---in the notation from the previous sections---be computed by
  \[
    \# z^{-1}(\cC) = \Poin_{\cC}(J(P);r-1)\,,
  \]
  see ~\cite[Proposition 5.2]{Ehrenborg-signed}.
  Since distributive lattices are modular, they admit $R$-labelings; see~\cite[Theorem~3.7]{bjorner-shellings}.
  Thus, applying~\Cref{thm:refinement-of-ber} for $y=r-1$ gives the first part of~\cite[Theorem 4.2]{Ehrenborg-signed}.
\end{remark}

\medskip
By examining the statistic from~\Cref{thm:combinatorial-interp}, we see that the
extended $\aaa\bbb$-index has symmetry among its coefficients. We encode this in
the following theorem, which generalizes the bivariate version
of~\cite[Theorem~A]{Maglione-Voll}. 

\begin{corollary}
\label{thm:symmetry-of-coeffs}
  Let~$P$ be an $R$-labeled poset of rank~$n$.
  Let moreover~$\cM$ be a maximal chain in~$P$ and let $E, E^c$ be complementary subsets of $\{1,\dots,n\} = E \cup E^c$.
  Then $\StatMon(\cM,E) = u_1\dots u_{n}$ and $\StatMon(\cM,E^c) = u_1^c\dots u_{n}^c$ are complementary words, \ie, $\{u_i,u_i^c\} = \{\aaa,\bbb\}$ for every $i \in \{1,\dots,n\}$.
  In particular, the coefficients of $y^\ell\mon$ and of $y^{n-\ell}\mon^c$ in $\extPsi(P;q,\aaa,\bbb)$ coincide,
  \[
    [y^\ell \mon]\ \extPsi(P;y,\aaa,\bbb) = [y^{n-\ell} \mon^c]\ \extPsi(P;y,\aaa,\bbb)\,,
  \]
  for complementary words~$\mon$ and~$\mon^c$ in $\aaa,\bbb$ and for $\ell\in
  \{0,\dots, n\}$.
\end{corollary}

We next turn toward the coarse flag Hilbert--Poincaré series introduced and
studied in~\cite{Maglione-Voll}. The numerator of this rational function is
defined in~\cite[Equation~(1.13)]{Maglione-Voll}, and we extend this definition
to graded posets via
\[
  \Num(P;y,t) = \sum_{\cC\text{ chain in } P\setminus\{\hat{0},\hat 1\}} \POIN[\{\hat 0\} \cup \cC]{P}{y}\cdot t^{\#\cC}(1-t)^{n-1-\#\cC}\, \in \Z[y,t]\,.
\]

\begin{example}
  Let~$\cL$ be the poset from \Cref{ex:face-poset2}.
  This poset has rank~$2$ and numerator polynomial
  \[
    \Num(P;y,t) = (1 + 3y + 2y^2)(1-t) + 3\cdot(1+y)^2t = 1 + 3y + 2y^2 + (2 + 3y + y^2)t\,.
  \]
\end{example}

The following corollary generalizes the property in \Cref{eq:ab01} to extended
$\aaa\bbb$-indices and relates them to coarse flag Hilbert--Poincaré series.
Recall from \Cref{eq:exabpoin}, the function $\iota$ that deletes the first
letter from every $\aaa\bbb$-monomial.

\begin{corollary}
\label{cor:enab0}
  For an $R$-labeled poset~$P$, we have
  \[
    \iota\big(\extPsi(P;y,\aaa,\bbb)\big) = \sum_{\cC\text{ chain in } P\setminus\{\hat 0,\hat 1\}} \Poin_{\{\hat 0\} \cup \cC}(P;y) \cdot \wtp_\cC\ \in \N[y]\langle \aaa,\bbb\rangle \,,
  \]
  where $\wtp_\cC = w_1\dots w_{n-1}$.
\end{corollary}

Specializing the equation from \Cref{cor:enab0}
via $\aaa \mapsto 1$ and $\bbb \mapsto t$
proves~\cite[Conjecture~E]{Maglione-Voll} and its generalization to $R$-labeled
posets. We collect this in the following.

\begin{corollary}
\label{thm:maglione-voll-conjecture}
  For an $R$-labeled poset~$P$, the coefficients of $\Num(P;y,t)$ are
  nonnegative.
\end{corollary}

Together with \Cref{eq:exabpoin}, we obtain
\begin{equation}
  \Poin(P;y) = [t^0]\ \Num(P;y,t). 
  \label{eq:poinfromnum}
\end{equation}
The substitutions in the previous corollaries show that \Cref{thm:combinatorial-interp} also gives analogous combinatorial interpretations for the coefficients of  $\iota\big(\extPsi(P;y,\aaa,\bbb)\big)$ and of $\Num(P;y,t)$. 

\begin{example}
  Recall that $\iota$ is the function that deletes the first letter of every
  $\aaa\bbb$-monomial. Thus, ignoring the first position in the computation in
  \Cref{ex:face-poset3} yields
  \[
    \iota\big(\extPsi(\cL;y,\aaa,\bbb)\big) = (1 + 3y + 2y^2)\aaa + (2 + 3y + y^2)\bbb\,.
  \]
  The substitution $\aaa \mapsto 1$ and $\bbb \mapsto t$ yields $\Num(\cL;y,t)$.
  Since $n=2$, we have
  \[
    [\aaa^{n-1}]\ \iota\big(\extPsi(\cL;y,\aaa,\bbb)\big) = [t^0]\ \Num(\cL;y,t) = \Poin(\cL;y)\,,
  \]
  also in agreement with \Cref{eq:exabpoin}.
\end{example}

\begin{remark}[Geometric semilattices]
  Note that~\cite[Conjecture~E]{Maglione-Voll} concerns all hyperplane arrangements (central and affine).
While the intersection posets of central hyperplane arrangements are geometric lattices and, thus, admit $R$-labelings~\cite[Example 3.8]{bjorner-shellings}, the intersection posets of affine arrangements are part of a more general family called \emph{geometric semilattices}, first explicitly studied by Wachs and Walker in~\cite{wachs-walker}.
A theorem of Ziegler shows that if $\cL$ is a geometric semilattice, then $\cL\cup\{\hat{1}\}$ admits an $R$-labeling~\cite[Theorem 2.2]{ziegler-shellability}.
Thus~\Cref{thm:combinatorial-interp} holds for intersection posets of affine arrangements (now with a formal unique maximal element included, which is consistent with the formulation in~\cite[Conjecture~E]{Maglione-Voll}).
\end{remark}

\begin{remark}[Implications for other zeta functions]\label{rem:implications-for-zetas}
The coarse flag Hilbert--Poincaré polynomial of a poset $P$ comes from a natural
specialization of its flag Hilbert--Poincaré series. The flag Hilbert--Poincaré
series is a rational function in $\mathbb{Q}[y](t_x \mid x \in P)$ given by 
\begin{align*}
  \mathsf{fHP}_P(y, \mathbf{t}) &= \sum_{\cC \text{ chain in } P\setminus \hat{0}} \POIN[\cC]{P}{y} \prod_{x\in \cC} \dfrac{t_x}{1-t_x}\,.
\end{align*}
The coarse flag Hilbert--Poincaré polynomial $\Num(P; y, t)$ is obtained by setting all the $t_x$ equal to $t$ and considering $(1 - t)^{\rank(P)}\mathsf{fHP}_P(y,t)$. 
Different specializations of $\mathsf{fHP}_P(y, \mathbf{t})$ yield other well-studied zeta functions like local Igusa zeta functions of hyperplane arrangements~\cite{Budur-Nero-Mustata}, motivic zeta functions of matroids from~\cite{Jensen-Kutler-Usatine}, and the conjugacy class counting zeta functions of certain group schemes defined in~\cite{Rossmann-Voll}.
Moreover, each of these is obtained from $\mathsf{fHP}_P(y, \mathbf{t})$ by a monomial substitution of the form $y = -p^{-1}$ and $t_x = p^{\lambda_x}t^{\mu_x}$ for some integers $\lambda_x$ and $\mu_x$, where $p$ is a prime and $t=p^{-s}$ for a complex variable~$s$; see \cite[Remark 1.3]{Maglione-Voll}.

In each of these settings, the poles and their orders play an important role in understanding how quickly the terms of these zeta functions grow.
Poles of zeta functions are challenging to compute in general---see for example \cite[Conjectures~2.3.1 \& 2.3.2]{Denef}, \cite[Conjecture~IV]{Rossmann}, and \cite[Question~1.8]{Rossmann-Voll}.
\Cref{thm:maglione-voll-conjecture} tells us that $\mathsf{fHP}_P(y,t)$ has exactly one pole at $t=1$ of order $\rank(P)$.
In particular, this result suggests there might be similar combinatorial interpretations of the numerators of zeta functions obtained from specializing $\mathsf{fHP}_P(y, \mathbf{t})$. 
\end{remark}

The specialization of $\Num(P;y,t)$ at $y=1$ was studied further for matroids
and oriented matroids by the second author and Kühne in~\cite{Kuhne-Maglione},
who showed $\Num(P;1,t)$ is the sum of $h$-polynomials of simplicial complexes
related to the chambers if $P$ is the lattice of flats of a real central
hyperplane arrangement. The following corollary gives a lower bound for the
coefficients of $\Num(P;1,t)$, whose proof also addresses the conjectured lower
bound in~\cite[Conjecture~1.4]{Kuhne-Maglione} generalized to the setting of
$R$-labeled posets---originally stated for geometric lattices.

\begin{corollary}
\label{thm:lower-bound-on-coefficients}
  Let~$P$ be an $R$-labeled poset of rank~$n$.
  The coefficients of $\Num(P;1,t)$ are bounded below by
   \[
     [t^k]\ \Num(P;1,t) \geq \binom{n-1}{k}\cdot \POIN{P}{1}\,.
   \]
\end{corollary}

\section{The combinatorial extended $\aaa\bbb$-index and quasisymmetric functions}
\label{sec:connections}

So far, we have only presented our results for $R$-labeled posets.
In this section, we consider more general graded posets with arbitrary edge labelings and explore connections to quasisymmetric functions.
We end with a Schur-positivity conjecture of the map~$\omega$ from \Cref{thm:refinement-of-ber} applied to a Schur function.

\medskip

\Cref{thm:combinatorial-interp} shows that the extended $\aaa\bbb$-index of an $R$-labeled poset has nonnegative coefficients.
Nonnegativity may fail, however, for posets that do not admit $R$-labelings.
For example, the weak order for the symmetric group $\mathfrak{S}_3$ (the hexagon poset) does \emph{not} admit an $R$-labeling, and its extended $\aaa\bbb$-index is
\begin{multline*}
  \aaa^{3} + \left(2 y + 1\right) \aaa^{2}\bbb + \left(2 y + 1\right) \aaa \bbb \aaa + \left(2 y^{2} - 1\right) \aaa \bbb^{2} + \\ 
  \left(-y^{3} + 2 y\right) \bbb \aaa^{2} + \left(y^{3} + 2 y^{2}\right) \bbb \aaa \bbb + \left(y^{3} + 2 y^{2}\right) \bbb^{2}\aaa + y^{3} \bbb^{3}\,.
\end{multline*}
Observe that this equals the $\omega$-evaluation $\omega\big( \aaa^{3} + \aaa^{2}\bbb + \aaa \bbb \aaa - \aaa \bbb^{2} \big)$ of its $\aaa\bbb$-index, in accordance with \Cref{conj:refinement-of-ber}.

\medskip

In this section, we define a combinatorial analogue of the extended $\aaa\bbb$-index and explore some of its properties.
This combinatorial analogue is \emph{manifestly positive} for all posets and closely related to certain quasisymmetric function identities.

\subsection{The combinatorial extended $\aaa\bbb$-index}

Using the right-hand side in \Cref{thm:combinatorial-interp}, we define the (\Dfn{combinatorial}) \Dfn{extended $\aaa\bbb$-index} of a finite edge-labeled graded poset~$P$ via
\[
  \combextPsi(P;y,\aaa,\bbb) = \sum_{(\cM,E)} y^{\#E}\cdot\StatMon(\cM,E) \ \in \N[y]\langle\aaa,\bbb\rangle\,.
\]
So if $P$ is $R$-labeled, then $\combextPsi$ and $\extPsi$ coincide, thanks to \Cref{thm:combinatorial-interp}.
For later reference, we also set $\combPsi(P;\aaa,\bbb) = \combextPsi(P;0,\aaa,\bbb)$ to be the (combinatorial) $\aaa\bbb$-index and $\combpullPsi(P;\aaa,\bbb) = \combextPsi(P;1,\aaa,\bbb)$ to be the (combinatorial) pullback $\aaa\bbb$-index.

\medskip

While this combinatorial description is in general not linked to the Poincaré polynomial, the proofs of \Cref{thm:refinement-of-ber,,cor:peakenumeration} still hold.
We record this in the following corollary.

\begin{corollary}
  \label{cor:comb-peakenumeration}
We have
\[
  \combextPsi(P;y,\aaa,\bbb)
  = \omega\big(\combPsi(P;\aaa,\bbb)\big) = \sum_{(\cM,S)} (1+y)^{\# S}\cdot y^{\bout(\cM,S)} \cdot \wt_S \,,
\]
where the sum ranges over all maximal chains~$\cM$ and all $\cM$-peak-covering $S \subseteq \{1,\dots,n\}$, where $\wt_S = w_1\dots w_n$ is given in \Cref{eqn:ext-weights}, and where $\bout(\cM,S)$ is the number of positions $i \notin S$ for which $u_i = \bbb$ in $\StatMon(\cM) = u_1\dots u_n$ as defined in \Cref{eq:statmon}.
\end{corollary}

In particular, $\combextPsi(P;y,\aaa,\bbb)$ is a polynomial in $\ccc_1 = \aaa+y\bbb, \ccc_2 = \bbb+y\aaa$ and $\ddd = \aaa\bbb+y\bbb\aaa+y\aaa\bbb+y^2\bbb\aaa$.
This means that $2\cdot \combextPsi(P;1,\aaa,\bbb)$ is an $\aaa\bbb$-analogue of the \emph{peak enumerator} from~\cite[Definition 7.1]{bergeron-mykytiuk-sottile-vanWilligenburg2000}; see \Cref{eq:Xicpull} below.

\subsection{Connections to quasisymmetric functions and $P$-partitions}\label{subsec:QSYM-connection}

In this section, we discuss some of the connections between the combinatorial extended $\aaa\bbb$-index and the theory of \emph{quasisymmetric functions}.
In order to make this precise, we first need to collect a few relevant definitions.
For more details, see~\cite[Section~1.4]{stembridge-p-partitions} and~\cite[Sections~6~\&~7]{bergeron-mykytiuk-sottile-vanWilligenburg2000}.

\medskip

Let $S = \{s_1 <\dots <s_k\}$ be a subset of $\{1,\dots,n\}$.
The \Dfn{monomial quasisymmetric function}~$M_S$ is the power series
\[
  M_S = \sum_{i_1 < i_2 < \cdots < i_k < i_{k+1}} x_{i_1}^{s_1} x_{i_2}^{s_2 - s_1} \cdots x_{i_k}^{s_k-s_{k-1}}x_{i_{k+1}}^{n+1   - s_k} \quad \in \mathbb{Q}[[x_1,x_2,x_3,\dots]]\,.
\]
Note that~$M_S$ is homogeneous of degree~$n+1$ and---although we surpress it in the notation---implicitly depends on~$n$.
The ring of \Dfn{quasisymmetric functions} $\QSym$ is the (linear) span of $M_{\bullet} = 1$ and all $M_S$ for $n \geq 0$.
Gessel introduced $\QSym$ to study $P$-partitions and gave a second basis
\begin{align}
  \label{eq:FSMT}
  F_{S} = \sum_{\{1,\dots,n\} \supseteq T \supseteq S} M_T \in \QSym\,,
\end{align}
which is related to the monomial quasisymmetric functions via inclusion-exclusion~\cite[Equation~2]{gessel}.
Following~\cite[Section~3]{Ehrenborg-Readdy-transforms}, we define a vector space isomorphism
\[
  \begin{array}{lcccl}
    \Xi : & \QQ\langle\aaa,\bbb\rangle &\longrightarrow& \QSym & \\[5pt]
          & \wt_T &\longmapsto& M_T &\,.
  \end{array}
\]
Comparing~\Cref{eq:FSMT} with the relation
\[
 \mon_S = \sum_{T \supseteq S} \wt_T \ \in \QQ\langle\aaa,\bbb\rangle
\]
yields
\[
  \Xi: \mon_S \longmapsto F_S\,.
\]
Using the isomorphism~$\Xi$, we can view the map~$\omega$ from \Cref{thm:refinement-of-ber} as a map
\[
  \omega : \QSym \longrightarrow \QSym \otimes \QQ[y]
\]
which is given by
\begin{equation}
\label{eq:quasiomega}
  F_S \mapsto \omega(F_S) = \Xi\big(\omega(\mon_S)\big) \,.
\end{equation}
As an example, consider $S = \{2\} \subseteq \{1,2,3\}$.
Then $\mon_S = \aaa\bbb\aaa$ and
\begin{align*}
  \omega(\mon_S)
  &= \big((1+y)\aaa\bbb + y(1+y)\bbb\aaa \big)(\aaa+y\bbb) \\
  &= (1+y)\aaa\bbb\aaa + y(1+y)\aaa\bbb\bbb + y(1+y)\bbb\aaa\aaa + y^2(1+y)\bbb\aaa\bbb\,,
\end{align*}
implying
\[
  \omega(F_S) = (1+y)F_{\{2\}} + y(1+y)F_{\{2,3\}}+ y(1+y)F_{\{1\}}+ y^2(1+y)F_{\{1,3\}}\in \QSym\otimes \QQ[y]\,.
\]
Next we use~$\Xi$ to turn relations concerning $\aaa\bbb$-indices into relations concerning quasisymmetric functions.

\medskip

Given a finite poset~$P$ together with an injective vertex labeling~$\gamma : P \rightarrow \N$, the \Dfn{weight enumerator} is the generating function of all \Dfn{$P$-partitions}.
As in~\cite[Equation 1.7]{stembridge-p-partitions}, the peak enumerator $\Gamma(P)$ is equal to the following quasisymmetric function
\begin{equation}
 \Gamma(P) = \sum_{\pi \in \LinExt(P)} F_{\Des(\pi)}
\label{eq:GammaF}
\end{equation}
where $\LinExt(P)$ is the set of linear extensions of~$P$ and
\[
 \Des(\pi) = \big\{ i \in \{1,\dots,n-1\} \mid \gamma(p_i) > \gamma(p_{i+1})\big\}
\]
is the set of \Dfn{descents} of a linear extension $\pi = [p_1,\dots,p_n]$.
Note that the set $\LinExt(P)$ of linear extensions can be identified with the set of maximal chains in its Birkhoff lattice~$J(P)$ by recording which element was added to the order ideal at each step of the chain; see \Cref{rem:jp-version}.
Given a vertex labeling~$\gamma$ of~$P$, we obtain an edge labeling~$\lambda_\gamma$ of~$J(P)$ with the property that $\lambda_\gamma$ is an $R$-labeling if and only if~$\gamma$ is a natural labeling of~$P$, \ie , if $p < q$ in~$P$ implies $\gamma(p) < \gamma(q)$.

\medskip

For example, \Cref{eq:GammaF} corresponds---via $\Xi$---to \Cref{eq:abstat}.
Now setting $y=0$ in the combinatorial extended $\aaa\bbb$-index and applying $\Xi$ gives
\[
  \Xi\big(\iota\combPsi(J(P);\aaa,\bbb)\big)
  = \Gamma(P)\,.
\]
where $\iota$ deletes the first letter of every monomial.
This first letter (which is equal to~$\aaa$ for every monomial in $\StatMon(\cM,\emptyset)$) is removed because it does not come from the comparison between labels along the chain as explained in \Cref{rem:bottom-element-toss-out}.
We illustrate this in the following example.

\begin{example}
  Consider the following pair of~$P$ and~$J(P)$.
  \[
    J\begin{pmatrix}\begin{tikzpicture}[scale=1]
    \node[invisivertex] (1)  at (0,0){$1$};
    \node[invisivertex] (2) at (1,0){$2$};
    \node[invisivertex] (3) at (.5,1){$3$};
    \draw (1) -- (3);
    \draw (2) -- (3);
  \end{tikzpicture}
  \end{pmatrix} = ~~
  \begin{matrix}
  \begin{tikzpicture}[scale=.75]
    \node[invisivertex] (empt)  at (1,0){$\emptyset$};
    \node[invisivertex] (1) at (0,1){$1$};
    \node[invisivertex] (2) at (2,1){$2$};
    \node[invisivertex] (12) at (1,2){$12$};
    \node[invisivertex] (123) at (1,3){$123$};
    \draw (empt) -- (1);
    \draw (empt) -- (2);
    \draw (1) -- (12);
    \draw (2) -- (12);
    \draw (12) -- (123);
  \end{tikzpicture}
  \end{matrix}~~.
  \]
  Then we have
  \[
    \Gamma(P) = F_{\emptyset} + F_{\{1\}} = M_{\emptyset} + 2M_{\{1\}} + M_{\{2\}} + 2M_{\{1,2\}}
  \]
  for the two linear extensions $[1,2,3]$ and $[2,1,3]$ of~$P$ with the respective descent sets.
  On the other hand, we obtain
  \[
    \iota\combPsi(J(P);\aaa,\bbb) = \iota(\aaa\aaa\aaa + \aaa\bbb\aaa) = \aaa\aaa+\bbb\aaa = (\aaa - \bbb)^2 + 2\bbb (\aaa - \bbb) + (\aaa - \bbb) \bbb + 2 \bbb \bbb\,.
  \]
\end{example}

Next we observe that the pullback $\aaa\bbb$-index is connected to the weight enumerator for \emph{enriched $P$-partitions} as defined in~\cite[Section 2]{stembridge-p-partitions}.
From~\cite[Equation~2.4]{stembridge-p-partitions}, the weight enumerator $\Delta(P)$ of an enriched $P$-partition can be expressed as a sum of quasisymmetric functions indexed by peak sets.
Setting $y=1$ in \Cref{cor:comb-peakenumeration} and comparing to \cite[Proposition 2.2]{stembridge-p-partitions} gives
\begin{equation}
\label{eq:Xicpull}
  \Xi : 2\cdot\iota\big(\combextPsi(J(P);1,\aaa,\bbb)\big) \mapsto \Delta(P)\,.
\end{equation}

In~\cite[Theorem 3.1(c)]{stembridge-p-partitions}, Stembridge shows that the weight enumerators~$\Delta(P)$ linearly span the peak algebra~$\Pi \subseteq \QSym$.
Concretely, $\Delta(P)$ can be written as a sum over maximal chains of $J(P)$.
For a maximal chain~$\cM$ in~$J(P)$ with corresponding linear extension $\pi = [p_1,\dots,p_n] \in \LinExt(P)$, let the peak set be
\[
  S = \big\{ i \in \{2,\dots, n-1\} \mid \gamma(p_{i-1}) < \gamma(p_i) > \gamma(p_{i+1}) \big\}\,.
\]
Following~\cite[Equation~5.3]{bergeron-mykytiuk-sottile-vanWilligenburg2000} and \cite[Proposition 2.2]{stembridge-p-partitions}, define  
\[
  \Theta_S = \sum_{S\subseteq T \subseteq S \cup (S+1)} 2^{\#T + 1} M_T\,,
\]
where $S+1 = \{ i+1 \mid i \in S\}$.
The \Dfn{peak algebra}~$\Pi \subseteq \QSym$ is the linear span of all $\Theta_S$ together with $M_\bullet = 1$.
Now combining \Cref{eq:thetadef} from page~\pageref{eq:thetadef} and \Cref{cor:comb-peakenumeration} gives
\begin{equation}
\label{eq:XiTheta}
  \Xi : 2\cdot\iota\big(\theta(\cM,1)\big) \mapsto \Theta_S\,.
\end{equation}

\subsection{Connections to Schur functions}
\label{sec:schur}

In~\cite[Equation~(1.8)]{stembridge-p-partitions}, Stembridge shows how to obtain (skew) Schur functions as $P$-partition enumerators of certain posets given in~\cite[Section~1.3]{stembridge-p-partitions}.
For a given (top-left aligned) Ferrers diagram of a partition $\lambda \vdash n$, the poset is obtained by ordering its cells from left to right and from bottom to top, and labeling them row by row with the numbers~$1$ through~$n$. The poset is therefore the set $\{1,\dots, n\}$ where $i\leq j$ if and only if the cell labeled by $i$ is northwest of the cell labeled by $j$.
Observe that this is usually not a natural labeling.

\begin{corollary}
  Let~$\lambda\vdash n$ be a partition identified with the above poset of its labeled Ferrers diagram.
  Then
  \[
    s_\lambda = \Xi\big(\iota\combPsi(J(\lambda);\aaa,\bbb)\big)\,.
  \]
\end{corollary}

\begin{example}
  Consider $\lambda = (2,2)\vdash 4$.
  The poset structure on its diagram is then
  
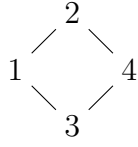
\[
  \begin{matrix}
  \begin{tikzpicture}[scale=.75]
    \node[invisivertex] (empt)  at (1,0){$3$};
    \node[invisivertex] (1) at (0,1){$1$};
    \node[invisivertex] (2) at (2,1){$4$};
    \node[invisivertex] (12) at (1,2){$2$};
    \draw (empt) -- (1);
    \draw (empt) -- (2);
    \draw (1) -- (12);
    \draw (2) -- (12);
  \end{tikzpicture}
  \end{matrix}
  \]
  and we obtain its two linear extensions $[3,1,4,2]$ and $[3,4,1,2]$.
  Thus,
  \begin{equation*}
    s_{[2,2]} = F_{\{1,3\}} + F_{\{2\}} 
              = \Xi(\bbb\aaa\bbb + \aaa\bbb\aaa)
              = \Xi\big(\iota(\aaa\bbb\aaa\bbb + \aaa\aaa\bbb\aaa)\big)
              = \Xi\big(\iota\combPsi(J(\lambda);\aaa,\bbb)\big)\,.
  \end{equation*}
  In addition, we can compute $\omega(s_{[2,2]})$ as given in \Cref{eq:quasiomega} by
  \begin{align*}
    \omega(s_{[2,2]})
    &= \Xi\big(\omega(\bbb\aaa\bbb+\aaa\bbb\aaa)\big) \\
    &= \Xi\big( (y^2+y)\aaa\aaa\bbb + (y^3+y^2+y+1)\aaa\bbb\aaa + (y^2+y)\aaa\bbb\bbb + \\
    &\hspace*{35pt}+(y^2+y)\bbb\aaa\aaa + (y^3+y^2+y+1)\bbb\aaa\bbb + (y^2+y)\bbb\bbb\aaa \big) \\
    &= (y^2+y) F_{\{3\}}+ (y^3+y^2+y+1)F_{\{2\}}+ (y^2+y)F_{\{2,3\}} \\
    &\hspace*{35pt}+(y^2+y)F_{\{1\}} + (y^3+y^2+y+1)F_{\{1,3\}} + (y^2+y)F_{\{1,2\}}\,.
  \end{align*}
  It turns out that this is a symmetric function, and its Schur expansion is
  \[
    \omega(s_{[2,2]}) = (y^2+y) s_{[2,1,1]} + (y^3+1) s_{[2, 2]} + (y^2+y) s_{[3, 1]}\,.
  \]
  The analogous calculation for $[3,2] \vdash 5$ yields
  \begin{align*}
    \omega(s_{[3,2]})
    &= (y^3+y^2)s_{[2, 1, 1, 1]} + (y^4+y^3+y^2+y)s_{[2, 2, 1]}\\
    &\hspace{25pt}+ (y^3+2y^2+y)s_{[3, 1, 1]}+ (y^3+y^2+y+1)s_{[3, 2]} + (y^2+y)s_{[4, 1]}\,.
  \end{align*}
\end{example}

These examples suggest the following conjecture\footnote{This conjecture was exhibited at the \emph{90th Séminaire Lotharingien de Combinatoire} in Bad Boll, Germany in September 2023 in collaboration with Darij Grinberg.}.

\begin{conjecture}
\label{conj:schurpos}
  For any partition $\lambda \vdash n$, the quasisymmetric function $\omega(s_\lambda)$ is symmetric and Schur positive.
  Specifically, for each $\mu\vdash n$, there exist $c_{\lambda}^{\mu}(y)\in \N[y]$ such that 
  \[
    \omega(s_\lambda) = \sum_{\mu \vdash n} c_{\lambda}^{\mu}(y)\cdot s_{\mu}\,.
  \] 
\end{conjecture}

After posting this paper to the arXiv, Ricky Liu proved the following theorem involving the Kronecker product (denoted by $\ast$) and Kronecker coefficients $g_{\lambda,\mu,\nu}$.

\begin{thm}[Liu]
  For any partition $\lambda \vdash n$,
  \[\omega(s_\lambda) = \sum_{k=0}^{n-1} (s_\lambda * s_{(n-k,1^k)}) y^k = \sum_{\mu \vdash n} c^\mu_{\lambda}(y)\cdot s_\mu,\] where \[c^\mu_{\lambda}(y) = \sum_{k=0}^{n-1} g_{\lambda,\mu,(n-k,1^k)} y^k \in \NN[y].\]
\end{thm}

The proof of this theorem is included as an appendix, see \Cref{thm:proof-of-conjecture} in \Cref{sec:appendix}.

\section{The combinatorial description of the extended $\aaa\bbb$-index}
\label{sec:proof-of-combinatorial-interp}

In this section we prove~\Cref{thm:combinatorial-interp} as well as~\Cref{thm:symmetry-of-coeffs,,thm:maglione-voll-conjecture,,thm:lower-bound-on-coefficients}.
First we use a property of $R$-labelings to rewrite the chain Poincaré polynomial as a sum over maximal chains (\Cref{lem:poincare-expansion}).
Then, in~\Cref{sec:inclusion-exclusion}, we use an inclusion-exclusion argument to give a combinatorial description of the coefficients of the extended $\aaa\bbb$-index.
The proof of~\Cref{thm:combinatorial-interp} and the corollaries are given in \Cref{sec:reinterp-coeffs} and come from reformulating this combinatorial description.

\medskip
Throughout this section, we fix an $R$-labeled poset~$P$ of rank~$n$ together
with an $R$-labeling~$\lambda$. We start by recalling the following rewriting of
the Möbius function in terms of maximal chains. 

\begin{lemma}[\!\!{\cite[Corollary 2.3]{bjorner-garsia-stanley}}]
\label{lem:moebiuscount}
  Let $X,Y \in P$ with $X \leq Y$.
  Then 
  \[
    (-1)^{\rank(Y) - \rank(X)}\mu(X,Y)
    =
    \#
    \big\{
    X = \cC_i \lessdot \cC_{i+1} \lessdot \dots \lessdot \cC_{j} = Y
    \mid
    \lambda(\cC_{i},\cC_{i+1}) > \dots > \lambda(\cC_{j-1},\cC_{j})
    \big\}\,.
  \]
\end{lemma}

We distinguish multisets from sets by using two curly brackets, and a multichain is a totally ordered multiset.
For $k\geq 0$ and a chain $\cC = \{\cC_1 < \cdots < \cC_k\}$ in $P$, say that a multichain $\cD=\multichain{\cD_1 \leq \cdots \leq \cD_k}$ in~$P$ \Dfn{interlaces}~$\cC$ if
\begin{equation}\label{eqn:interlaced}
  \cC_1 \leq \cD_1 \leq \cC_2 \leq \cD_2 \leq \dots \leq \cC_k \leq \cD_k\,.
\end{equation}
Note that we allow~$\cD$ to be a multichain only because we could have $\cD_{i-1} = \cC_i = \cD_i$ for some position~$i$.
We call such a pair $(\cC,\cD)$ an \Dfn{interlacing pair}. 
For a chain $\cC$ in~$P$, let 
\[ 
  \Int(\cC) = \{ \cD ~|~ (\cC, \cD) \text{ is an interlacing pair}\}, 
\] 
the set of multichains interlacing~$\cC$.
For $(\cC,\cD)$ interlacing, we write
\begin{align*}
  \IRank(\cC,\cD)
  &= \big\{\rank(\cC_1)+1,\dots,\rank(\cD_1) \big\} \cup \dots \cup \big\{\rank(\cC_k)+1,\dots,\rank(\cD_k) \big\} \\
  &= \big\{ r \in \{1,\dots,n\} \ \big\vert\ \rank(\cC_i) < r \leq \rank(\cD_i) \text{ for some position } i \big\}\,,
\end{align*}
and we denote the cardinality of this set by
\[
  \irank(\cC,\cD) = \# \IRank(\cC,\cD) = \sum_{i=1}^k\rank(\cD_i) - \sum_{i=1}^k \rank(\cC_i)\,.
\]
A maximal chain $\cM = \{ \cM_0 \lessdot \cM_1 \lessdot \dots \lessdot \cM_n\}$ in~$P$ \Dfn{decreases along} the interval $[\cM_i,\cM_j]$ if
\[
  \lambda(\cM_i, \cM_{i+1}) > \dots > \lambda(\cM_{j-1}, \cM_{j})\,.
\]
Similarly, we say that $\cM$ \Dfn{weakly increases along} $[\cM_i,\cM_j]$ if
\[
  \lambda(\cM_i, \cM_{i+1}) \leq \dots \leq \lambda(\cM_{j-1}, \cM_{j})\,.
\]
We say $\cM$ is \Dfn{decreasing} (resp.\ \Dfn{weakly increasing}) if $\cM$ is decreasing (resp.\ weakly increasing) along $[\cM_0, \cM_n] = P$. Let $\incdec(\cC,\cD)$ be the set of maximal chains~$\cM$ in~$P$ that refine the underling chain from~\Cref{eqn:interlaced} which
\begin{itemize}
  \item decrease along all the intervals of the form $[\cC_i,\cD_i]$, and
  \item weakly increase along all the intervals of the forms $[\hat{0}, \cC_1]$, $[\cD_i,\cC_{i+1}]$ and $[\cD_k, \hat{1}]$.
\end{itemize}
We say that such a maximal chain $\cM \in \incdec(\cC,\cD)$ is \Dfn{alternating} with respect to $(\cC,\cD)$.

\begin{proposition}
\label{lem:poincare-expansion}
  Let~$P$ be an $R$-labeled poset, and let~$\cC$ be a chain in~$P$.
  Then
  \[
    \POIN[\cC]{P}{y} = \sum_{\ell \geq 0} \left( y^\ell\cdot \sum_{\substack{\cD \in \Int(\cC) \\ \irank(\cC,\cD)=\ell}} \# \incdec(\cC,\cD)\right).
  \]
\end{proposition}

\begin{proof}
  We first rewrite the chain Poincaré polynomial, and then invoke~\Cref{lem:moebiuscount}. For $\cC = \{\cC_1 < \cdots < \cC_k\}$, we have 
  \begin{align*}
    \POIN[\cC]{P}{y} & = \prod_{i=1}^{k} \POIN{[\cC_i,\cC_{i+1}]}{y}
    =  \prod_{i=1}^{k} \left(\sum_{X \in [\cC_i,\cC_{i+1}]} \mu(\cC_i,X)\cdot (-y)^{\rank(X)-\rank(\cC_i)}\right)\\
    & =  \sum_{\ell\geq 0} \left(y^\ell\cdot \sum_{\substack{\cD \in \Int(\cC) \\\irank(\cC,\cD) = \ell}}\prod_{i=1}^{k} \left((-1)^{\rank(\cD_i, \cC_i)}\mu(\cC_i,\cD_i)\right)\right)\,.
  \end{align*}
  We may interpret $(-1)^{\rank(\cD_i, \cC_i)}\mu(\cC_i,\cD_i)$ via maximal chains from $\cC_i$ to $\cD_i$ using~\Cref{lem:moebiuscount}.
  From the definition of an $R$-labeling, each interval $[\cD_i,\cC_{i+1}]$ has a unique weakly increasing chain.
  By gluing together weakly increasing and decreasing pieces, we finally obtain, for $\cD\in \Int(\cC)$,
  \[
  \prod_{i=1}^{k} \left((-1)^{\rank(\cD_i, \cC_i)}\mu(\cC_i,\cD_i)\right) = \# \incdec(\cC,\cD) . \qedhere 
  \]
\end{proof}

\subsection{An inclusion-exclusion construction}
\label{sec:inclusion-exclusion}

The inclusion-exclusion construction starts with a set $\PIE_\ell(S)$ associated to a subset $S\subseteq \{0,\dots,n-1\}$ and a parameter~$\ell \geq 0$.
It is given by
\[
\PIE_\ell(S) =
\left\{
(\cC,\cD,\cM)\
\middle\vert\
\begin{matrix}
  (\cC,\cD) \text{ interlacing}\\
  \Rank(\cC) = S,\  \irank(\cC,\cD) = \ell\\
  \cM \in  \incdec(\cC,\cD)
\end{matrix}\
\right\}.
\]
In other words, we consider triples $(\cC,\cD,\cM)$ carrying the following data:
a chain~$\cC$ in~$P$ with the prescribed ranks~$\Rank(\cC) = S$,
a multichain $\cD$ interlacing~$\cC$ with $\irank(\cC,\cD) = \ell$, and
a maximal chain $\cM$ that is alternating with respect to the pair $(\cC,\cD)$.
The following lemma is immediate from the definition.

\begin{lemma}
\label{lem:AIncDec}
  For all $S \subseteq \{0, \dots, n-1\}$ and $\ell\geq 0$, we have
  \[
  \# \PIE_\ell(S) = \sum_{(\cC,\cD)} \# \incdec(\cC,\cD),
  \]
  where the sum ranges over all interlacing pairs $(\cC,\cD)$ with $\Rank(\cC)\! =\! S$ and $\irank(\cC,\cD)\! =\! \ell$.
\end{lemma}

\begin{remark}
  Given a graded poset~$P$ of rank~$n$ and a subset $S \subseteq \{0,\dots,n-1\}$, let $\alpha(S)$ denote the number of chains in~$P$ with prescribed ranks $S$, and let $\beta(T)$ be the signed sum $\sum_{S \subseteq T} (-1)^{\#(T\setminus S)} \alpha(S)$.
  The $\aaa\bbb$-index of~$P$ can then be expressed as
  \[
    \Psi(P;\aaa,\bbb) = \sum_{T\subseteq \{0,\dots,n-1\}} \beta(T) \cdot \StatMon_T
  = \sum_{T\subseteq \{0,\dots,n-1\}} \left(\sum_{S \subseteq T} (-1)^{\#(T\setminus S)} \alpha(S)\right)\cdot \StatMon_T,
  \]
  where $\StatMon_T = w_0\dots w_{n-1}$ with $w_i = \bbb$ if $i \in T$ and $w_i = \aaa$ if $i \notin T$.
  In this section we define an analogue of~$\alpha(S)$ for the extended $\aaa\bbb$-index, and use an inclusion-exclusion construction to obtain the analogue of~$\beta(S)$.
  This describes, in particular, how we use $R$-labelings since our analogue of~$\alpha(S)$ uses this machinery to interpret the Möbius function via chains as described in \Cref{lem:moebiuscount}.
\end{remark}

Define the embedding $\varphi_{S, T} : \PIE_\ell(S) \hookrightarrow \PIE_\ell(T)$ for $S \subseteq T \subseteq \{0,\dots,n-1\}$ given by
\[
  \varphi_{S,T}(\cC,\cD,\cM) = \big(\cC\cup\{\cM_r \mid r \in T \setminus S\},\ \cD\cup\{\cM_r \mid r\in T \setminus S\},\ \cM \big)\,.
\]
It is immediate that $\cM \in \incdec(\cC,\cD)$ implies $\cM \in \incdec(\cC\cup\{\cM_r\}, \cD\cup\{\cM_r\})$ for $r \notin \Rank(S)$, so this embedding is well-defined.
These maps are compatible with one another in the sense that $\varphi_{T,U} \circ \varphi_{S,T} = \varphi_{S,U}$ for $S \subseteq T \subseteq U$.
For $T\subseteq \{0,\dots,n-1\}$, this allows us to define the set
\[
  \topPIE_\ell(T) = \PIE_\ell(T) \big\backslash \bigcup_{S\subsetneq T} \varphi_{S,T}\big(\PIE_\ell(S)\big).
\]

\begin{proposition}
\label{lem:A-B-expansion}
  For $T\subseteq \{0,\dots,n-1\}$ and $\ell\geq 0$, we have
  \[
  \# \topPIE_\ell(T)
  = \sum_{S \subseteq T} (-1)^{\#T - \#S}
  \sum_{(\cC,\cD)} \# \incdec(\cC,\cD)\,,
  \]
  where the inner sum ranges over all interlacing pairs $(\cC,\cD)$ in~$P$ for which $\Rank(\cC) = S$ and $\irank(\cC,\cD) = \ell$.
\end{proposition}

\begin{proof}
  We have
  \[
  \#\topPIE_\ell(T) = \# \PIE_\ell(T) - \# \bigcup_{S\subsetneq T} \varphi_{S,T}\big(\PIE_\ell(S)\big)
   = \sum_{S \subseteq T} (-1)^{\#T - \#S} \# \PIE_\ell(S)\,,
 \]
 where the second equality follows from the principle of inclusion-exclusion.
 The statement follows from \Cref{lem:AIncDec}.
\end{proof}

We are now ready to give a first combinatorial description for the coefficients of the extended $\aaa\bbb$-index, which are given by the cardinalities of the sets $\topPIE_\ell(T)$.

\begin{thm}
\label{thm:toppiecount}
  For all $T\subseteq \{0, \dots, n-1\}$ and $\ell\geq 0$, 
  \[
    [y^\ell\cdot \mon_T]\ \extPsi(P;y,\aaa,\bbb) = \# \topPIE_\ell(T)\,,
  \]
  where $\mon_T = m_0\dots m_{n-1}$ with $m_i = \bbb$ if $i \in T$ and $w_i = \aaa$ if $i \notin T$.
\end{thm}

\begin{proof}
  Rewriting the weight function in the definition of $\extPsi(P;y,\aaa,\bbb)$ gives
    \begin{align*}
      \extPsi(P;y,\aaa,\bbb) & = \sum_{\cC \text{ chain in }P\setminus \{\hat{1}\}} \POIN[\cC]{P}{y} \cdot \wt_\cC\\[5pt]
      = & \sum_{\cC \text{ chain in }P\setminus \{\hat{1}\}}~\sum_{\substack{T \subseteq \{0,\dots,n-1\} \\ \Rank(\cC) \subseteq T}} (-1)^{\#T\setminus\Rank(\cC)} \POIN[\cC]{P}{y}\cdot \mon_T.
    \end{align*}
  By specifying the ranks of these chains and rearranging the order of summation, we obtain
  \begin{align*}
    \extPsi(P;y,\aaa,\bbb)
    = & \sum_{S\subseteq \{0,\dots,n-1\}}~\sum_{\substack{\cC \text{ chain in }P\\\Rank(\cC) = S}}~\sum_{\substack{T \subseteq \{0,\dots,n-1\} \\ S \subseteq T}} (-1)^{\#T\setminus S} \POIN[\cC]{P}{y} \cdot \mon_T\\[5pt]
    = & \sum_{S \subseteq T \subseteq \{0,\dots,n-1\}}~\sum_{\substack{\cC \text{ chain in }P\\\Rank(\cC) = S}} (-1)^{\# T\setminus S} \POIN[\cC]{P}{y}\cdot \mon_T\,.
  \end{align*}
  Applying \Cref{lem:poincare-expansion,,lem:A-B-expansion} yields
  \begin{align*}
    \extPsi(P;y,\aaa,\bbb) &= \sum_{\ell\geq 0} y^{\ell} \sum_{T\subseteq \{0,\dots,n-1\}} \left(\sum_{S \subseteq T} (-1)^{\#(T\setminus S)} \sum_{(\cC,\cD)} \#\incdec(\cC,\cD) \right) \cdot \mon_T
    \\[5pt] &=
    \sum_{\ell\geq 0} y^{\ell} \sum_{T\subseteq \{0,\dots,n-1\}}\#\topPIE_\ell(T) \cdot \mon_T . \qedhere 
  \end{align*}
  \end{proof}

  \begin{example}
    Below we compute $\PIE_\ell(T)$ and $\topPIE_\ell(T)$ for the poset $\cL$ from~\Cref{ex:face-poset2}.
    We record a triple $(\cC,\cD,\cM) \in \PIE_\ell(T)$ as follows:
    a maximal chain~$\cM = \{\hat 0 \lessdot \alpha_i \lessdot \hat 1\}$ is recorded by~$\alpha_i$,~$\cC$ is given by~$\cM$ and the property $\Rank(\cC) = T$, and~$\cD$ is finally given by $\IRank(\cC,\cD) = E \subseteq \{1,2\}$.
    \begin{center}
      \begin{tabular}{|c|c||c|c|c|}
      \hline
      & & & & \\[-10pt]
      T & $\ell$ & $E$ & $\PIE_\ell(T)$ & $\topPIE_\ell(T)$\\[2pt]
      \hline
      & & & & \\[-10pt]
      $\emptyset$ &  $0$     & $\emptyset$ & $\alpha_1$ & $\alpha_1$\\
      $\{0\}$     &  $0$     & $\emptyset$ & $\alpha_1$ & $\emptyset$ \\
      $\{1\}$     &  $0$     & $\emptyset$ & $\alpha_1,\alpha_2,\alpha_3$ & $\alpha_2,\alpha_3$ \\
      $\{ 0,1 \}$ &  $0$     & $\emptyset$ & $\alpha_1,\alpha_2,\alpha_3$ & $\emptyset$ \\
       \hline
       & & & \\[-10pt]
      $\emptyset$ &  $1$     &    -         & $\emptyset$ & $\emptyset$\\
      $\{0\}$     &  $1$     & $\{1\}$     & $\alpha_1,\alpha_2,\alpha_3$ & $\alpha_1,\alpha_2,\alpha_3$ \\
      $\{1\}$     &  $1$     & $\{2\}$     & $\alpha_1,\alpha_2,\alpha_3$ & $\alpha_2,\alpha_3$ \\
      $\{ 0,1 \}$ &  $1$     & $\{1\},\{2\}$     & $\alpha_1,\alpha_2,\alpha_3$ & $\emptyset$ \\
       \hline
       & & & \\[-10pt]
      $\emptyset$ &  $2$     & -     & $\emptyset$ & $\emptyset$\\
      $\{0\}$     &  $2$     & $\{1,2\}$     & $\alpha_2,\alpha_3$ & $\alpha_2,\alpha_3$ \\
      $\{1\}$     &  $2$     & -     & $\emptyset$ & $\emptyset$  \\
      $\{ 0,1 \}$ &  $2$     & $\{1,2\}$     & $\alpha_1,\alpha_2,\alpha_3$ & $\alpha_1$ \\
      \hline
      \end{tabular}
    \end{center}
  \end{example}

\subsection{Reinterpreting the coefficients}
\label{sec:reinterp-coeffs}
Next we reinterpret $\topPIE_\ell(T)$ in order to prove \Cref{thm:combinatorial-interp}.
We show, for a given maximal chain~$\cM$ in~$P$ and a set~$E\subseteq \{1,\dots,n\}$, that there is a unique interlacing pair~$(\cC,\cD)$ and a unique set~$T\subseteq\{0,\dots,n-1\}$ such that
\begin{equation}
  \IRank(\cC,\cD) = E\quad \text{and}\quad (\cC,\cD,\cM) \in \topPIE_{\#E}(T)\,.
  \label{eq:uniquecond}
\end{equation}
For a set $E\subseteq\{1,\dots,n\}$, let $I_E,J_E \subseteq \{1,\dots,n\}$ denote the set of ranks where intervals of consecutive elements outside and inside of~$E$ end, respectively.
In symbols,
\[
  I_E = \big\{i \in \{1,\dots,n\} \mid i \not\in E, i+1\in E \big\}
  \quad\text{and}\quad
  J_E = \big\{i \in \{1,\dots,n\} \mid i \in E, i+1\not\in E \big\}\,.
\]
The next lemma now follows immediately from unpacking definitions.

\begin{lemma}\label{lem:IRank-characterization}
  Let $E \subseteq \{1,\dots,n\}$.
  An interlacing pair $(\cC,\cD)$ satisfies $\IRank(\cC,\cD) = E$ if and only if
  \[
    I_E = \Rank(\cC) \setminus \big(\Rank(\cC)\cap\Rank(\cD)\big), \quad J_E = \Rank(\cD) \setminus \big(\Rank(\cC)\cap\Rank(\cD)\big)\,.
  \]
\end{lemma}

For a maximal chain~$\cM$, a set~$E\subseteq \{1,\dots,n\}$, and a subset $R\subseteq \{0,\dots,n\} \setminus I_E$, set
\begin{equation}\label{eqn:C_R-D_R}
  \cC_R = \{ \cM_i \mid i \in I_E \cup R\}\qquad \text{and}\qquad
  \cD_R = \multichain{\cM_i \mid i \in J_E \sqcup R}\,,
\end{equation}
where $J_E \sqcup R$ is a \emph{multiset} union. For a maximal chain $\cM$ in
$P$, define
\[ 
  \Int_{\cM,E} = \big\{ (\cC,\cD) \mid \cD \in \Int(\cC),\ \IRank(\cC,\cD) = E,\ \cM \in \incdec(\cC, \cD) \big\}\, .
\]
Recall from \Cref{sec:inclusion-exclusion} that the set $\PIE_\ell(S)$ consists of triples $(\cC,\cD,\cM)$ for interlacing pairs $(\cC,\cD)$ such that $\Rank(\cC) = S$ and~$\cM$ is alternating with respect to $(\cC,\cD)$.
Our next step is to use $\cC_R$ and $\cD_R$ to rewrite $\PIE_\ell(S)$.
To that end, we identify the subsets $R \subseteq\{0,\dots,n\}$ for which a given maximal chain $\cM$ is alternating with respect to $(\cC_R,\cD_R)$.

\begin{lemma}
  \label{prop:Tdescription}
  Let~$\cM$ be a maximal chain, and let $E \subseteq \{1,\dots,n\}$.
  Given the monomial $\StatMon(\cM) = u_1 \cdots u_n$, let $\TT{\cM}{E}$ be the set of
  indices $i \in \{0,\dots,n-1\}$ given by
  \[
    \TT{\cM}{E} = \big\{ i \mid i, i+1 \in E,\ u_{i+1} = \aaa\big\} \cup \big\{ i \mid i, i+1 \notin E,\ u_{i+1} = \bbb\big\}\,.
  \]
  Then we have 
  \[
    \Int_{\cM,E} = \big\{ (\cC_R,\cD_R) \mid \TT{\cM}{E} \subseteq R \subseteq \{0,\dots,n\} \setminus I_E \big\}\,.
  \]
\end{lemma}

\begin{proof}
  By \Cref{lem:IRank-characterization}, an interlacing pair $(\cC, \cD)$
  satisfies $\IRank(\cC, \cD) = E$ if and only if there exists $R\subseteq
  \{0,\dots, n\}\setminus I_E$ such that $(\cC, \cD)= (\cC_R, \cD_R)$. Thus, it
  suffices to identify those sets~$R$ for which the given maximal chain~$\cM$ is
  alternating with respect to~$(\cC_R,\cD_R)$. That is, $\lambda$
  decreases along all intervals of $\cM$ of the form $[\cC_i,\cD_i]$ and weakly
  increases along intervals of the form $[\hat{0},\cC_1]$, $[\cD_i,\cC_i]$ and
  $[\cD_i,\hat{1}]$, where $\cC=\cC_R$ and $\cD=\cD_R$. Recall that
  $u_{i+1}=\aaa$ if $\lambda(\cM_{i-1}, \cM_{i})\leq \lambda(\cM_{i},
  \cM_{i+1})$ and $u_{i+1}=\bbb$ otherwise. Let $i\in \TT{\cM}{E}$, and assume
  $i\notin R$.
  If $i,i+1 \in E$, then $\lambda(\cM_{i-1}, \cM_{i})\leq
  \lambda(\cM_{i}, \cM_{i+1})$, so $\cM$ is not decreasing along
  $[\cC_i,\cD_i]$.
  The case where $i,i+1 \notin E$ is analogous.
  Hence,~$\cM\in \incdec(\cC_R,\cD_R)$ if and only if $\TT{\cM}{E} \subseteq R$,
  so the lemma follows.
\end{proof}

\begin{example}
  Recall the poset $\cL$ from \Cref{ex:face-poset2}. We give, for each maximal
  chain in $\cL$ and each subset of $\{0,1,2\}$, the unique pair of interlacing
  chains satisfying \Cref{eq:uniquecond}.
  \begin{center}
    \begin{tabular}{|c|c|c|c|}
    \hline
    & & &\\[-10pt]
    & $\hat 0 \underset{\color{magenta}1}{\lessdot} \alpha_1 \underset{\color{magenta}2}{\lessdot} \hat 1$
    & $\hat 0 \underset{\color{magenta}2}{\lessdot} \alpha_2 \underset{\color{magenta}1}{\lessdot} \hat 1$ &
    $\hat 0 \underset{\color{magenta}3}{\lessdot} \alpha_3 \underset{\color{magenta}1}{\lessdot} \hat 1$ \\[5pt]
    \hline
    & & & \\[-10pt]
    $\{\}$        & $\{\},\{\}$ & $\{\alpha_2\},\{\alpha_2\}$ & $\{\alpha_3\},\{\alpha_3\}$ \\
    $\{ 0 \}$     & $\{\hat{0}\},\{\alpha_1\}$ & $\{\hat{0}\},\{\alpha_2\}$ & $\{\hat{0}\},\{\alpha_3\}$ \\
    $\{ 1 \}$     & $\{\alpha_1\},\{\hat{1}\}$ & $\{\alpha_2\},\{\hat{1}\}$ & $\{\alpha_3\},\{\hat{1}\}$ \\
    $\{ 0,1 \}$   & $\{\hat{0}<\alpha_1\},\{\alpha_1<\hat{1}\}$ & $\{\hat{0}\},\{\hat{1}\}$ & $\{\hat{0}\},\{\hat{1}\}$ \\
    \hline
    \end{tabular}
  \end{center}
\end{example}

The next step is to translate \Cref{prop:Tdescription} into a statement about $\PIE_\ell(S)$, which we then use to simplify the description of $\topPIE_\ell(T)$.

\begin{proposition}
\label{prop:PIEdescription}
  Let $S\subseteq\{0,\dots,n-1\}$ and $\ell\geq 0$.
  We have the following decomposition of $\PIE_\ell(S)$ into disjoint subsets
  \[
    \PIE_\ell(S) = \bigcup_{(\cM,E)} \big\{ (\cC,\cD,\cM) \mid (\cC,\cD) \in \Int_{\cM,E} \text{ and } \Rank(\cC) = S \big\}\,,
  \]
  where the disjoint union ranges over all maximal chains~$\cM$ and all subsets $E \subseteq \{1,\dots,n\}$ of cardinality $\#E = \ell$.
  Moreover, 
  \[
    \topPIE_{\ell}(T) = \left\{ \left(\cC_{\TT{\cM}{E}},\cD_{\TT{\cM}{E}},\cM\right) ~\middle|~ \begin{array}{c} \cM \text{ maximal chain},\ E \subseteq \{1,\dots,n\}, \\ \#E = \ell,\ \Rank(\cC_{\TT{\cM}{E}}) = T \end{array} \right\}\,.
  \]
\end{proposition}

\begin{proof}
  Let~$\cM$ be a maximal chain, $E\subseteq\{1,\dots,n\}$ of cardinality $\# E = \ell$, and $S \subseteq \{0,\dots,n-1\}$.
  We then write
  \[
    \idInt_{\cM,E}(S) = \big\{ (\cC,\cD,\cM) \mid (\cC,\cD) \in \Int_{\cM,E} \text{ and } \Rank(\cC) = S \big\}\,,
  \]
  and observe that both $\idInt_{\cM,E}(S) \subseteq \PIE_\ell(S)$ and $\idInt_{\cM,E}(S) \cap \idInt_{\cM,E'}(S) = \emptyset$ for different subsets $E$ and $E'$ of $\{1,\dots,n\}$.
  Moreover, for $(\cC,\cD,\cM)\in \PIE_\ell(S)$ with $E = \IRank(\cC,\cD)$, we have that $(\cC,\cD,\cM) \in \idInt_{\cM,E}$.
  This proves the first claim.

  The set $\topPIE_\ell(T)$ comprises the elements in $\PIE_\ell(T)$ that do not appear in $\PIE_\ell(S)$ for any proper subset $S\subsetneq T$.
  From \Cref{prop:Tdescription}, it follows that
  \[
    \idInt_{\cM,E}(S) = \big\{ (\cC_R,\cD_R,\cM) \mid \TT{\cM}{E} \subseteq R \subseteq \{0,\dots,n\} \setminus I_E \text{ and } \Rank(\cC_R) = S\big\}\,.
  \]
  The elements in $\PIE_\ell(T)$ not in $\PIE_\ell(S)$ for $S\subsetneq T$ are, thus, exactly the elements of the form $(\cC_R,\cD_R,\cM)$ for some maximal chain~$\cM$ and some subset $E \subseteq \{1,\dots,n\}$ such that $\Rank(\cC_R) = T$ and $R = \TT{\cM}{E}$.
  This yields the proposed description of~$\topPIE_{\ell}(T)$.
\end{proof}

We finally describe the cardinality of the set $\topPIE_\ell(T)$ in terms of the statistic $\StatMon(\cM,E)$ given in \Cref{sec:mainresults}.
Recall for a subset $T \subseteq \{0,\dots,n-1\}$, the monomial $\StatMon_T = w_0\dots w_{n-1}$ in $\aaa,\bbb$ is given by $w_i = \bbb$ if $i \in T$ and $w_i = \aaa$ if $i \notin T$.

\begin{corollary}
\label{prop:topPIEStat}
  Let $\ell \geq 0$, and let $T \subseteq \{0,1,\dots,n-1\}$.
  Then $\#\topPIE_\ell(T)$ is the number of pairs $(\cM,E)$ of a maximal chain~$\cM$ and a subset $E \subseteq \{1,\dots,n\}$ of cardinality $\#E = \ell$ such that $\StatMon(\cM,E) = \StatMon_T$.
\end{corollary}

\begin{proof}
  \Cref{prop:PIEdescription} shows that for every pair $(\cM,E)$, the triple $(\cC_{R},\cD_{R},\cM)$ is contained in $\topPIE_{\#E}(T)$ for $T = \Rank(\cC_{R})$ and $R=\TT{\cM}{E}$.
  The definition of $\cC_{R}$ yields that~$T$ is the set of positions in $\{0,\dots,n-1\}$ given by
  \[
    \big\{ i \mid i \notin E,\ i+1 \in E \big\}
       \cup \big\{ i \mid i, i+1 \in E,\ u_{i+1} = \aaa\big\} 
       \cup \big\{ i \mid i, i+1 \notin E,\ u_{i+1} = \bbb\big\}\,,
  \]
  where $\StatMon(\cM) = \StatMon(\cM,\emptyset) = u_1\cdots u_n$.
   The description of $\StatMon(\cM,E)$ in \Cref{sec:mainresults} can be easily seen to be equivalent to $\StatMon(\cM,E) = v_1\dots v_n$ with
   \begin{equation*}
     \begin{array}{ll}
             v_{i+1} = \aaa &\text{ if } \begin{cases}
                                 i\in E,\quad i+1\notin E, \quad \text{or} \\
                                 i \notin E, \quad i+1\notin E, \quad u_{i+1} = \aaa \quad \text{ or} \\
                                 i\in E, \quad i+1\in E, \quad u_{i+1} = \bbb\,,
                               \end{cases} \\ \\
             v_{i+1} = \bbb &\text{ if } \begin{cases}
                                 i\notin E,\quad i+1\in E, \quad \text{or} \\
                                 i \in E, \quad i+1\in E, \quad u_{i+1} = \aaa \quad \text{ or} \\
                                 i\notin E, \quad i+1\not\in E, \quad u_{i+1} = \bbb
                               \end{cases}
     \end{array}
   \end{equation*}
   for~$i \in \{0,\dots,n-1\}$.
   Comparing these two descriptions of~$T$ and of $\StatMon(\cM,E)$ proves that $\StatMon_T = \StatMon(\cM,E)$.
\end{proof}

\begin{proof}[Proof of~\Cref{thm:combinatorial-interp}]
  Combining \Cref{prop:topPIEStat} and \Cref{thm:toppiecount} shows that the coefficient of $y^\ell \StatMon_T$ in $\extPsi(P;y,\aaa,\bbb)$ equals
  \[
      \#\big\{ (\cM,E) \mid \cM \text{ maximal chain},\ E \subseteq \{1,\dots,n\},\ \#E = \ell,\ \StatMon(\cM,E) =  \StatMon_T \big\}\,.
  \]
  In other words, we have
  \[
    \extPsi(P;y,\aaa,\bbb) = \sum_{(\cM,E)} y^{\#E}\cdot\StatMon(\cM,E)\,,
  \]
  where the sum ranges over all maximal chains~$\cM$ and all subsets $E \subseteq \{1,\dots,n\}$.
\end{proof}

\begin{proof}[Proof of~\Cref{thm:symmetry-of-coeffs}]
  For a given chain $\cM$, exchanging $E\subseteq \{1,\dots,n\}$ with its complement $E^c = \{1,\dots,n\}\setminus E$ also exchanges roles of $\aaa$ and $\bbb$ in $\StatMon(\cM,E) = u_1\dots u_n$ to $\StatMon(\cM,E^c) = u^c_1\dots u^c_n$, meaning that $\{u_i,u_i^c\} = \{\aaa,\bbb\}$ for every position~$i$.
  The result now follows from \Cref{thm:combinatorial-interp}.
\end{proof}

\begin{proof}[Proof of~\Cref{cor:enab0}]
  Setting
  \[
    \extPsi^\circ(P;y,\aaa,\bbb) = \sum_{\cC\text{ chain in } P\setminus\{\hat 0,\hat 1\}} \POIN[\{\hat 0\} \cup \cC]{P}{y} \cdot \wtp_\cC\ \in \N[y]\langle \aaa,\bbb\rangle \,,
  \]
  we show that
  \begin{equation}
    \extPsi^\circ(P;y,\aaa,\bbb) = \iota\big(\extPsi(P;y,\aaa,\bbb)\big)\,.
    \label{eq:iotaextpsi}
  \end{equation}
  All previously given arguments remain valid when only considering those sets $\PIE_\ell(S)$ and $\topPIE_\ell(T)$ that contain~$0$.
  For $T \subseteq \{1,\dots,n-1\}$, we set
  \begin{equation}
    \topPIE^\circ_\ell(T) = \PIE_\ell(T\cup\{0\}) \big\backslash \bigcup_{S\subsetneq T} \varphi_{S,T}\big(\PIE_\ell(S\cup\{0\})\big)\,.
    \label{eq:toppiecirc}
  \end{equation}
  It is immediate that the analogue of \Cref{thm:toppiecount} holds, meaning that
  \begin{equation}
    [y^\ell \mon_T]\ \extPsi^\circ(P;y,\aaa,\bbb) = \# \topPIE^\circ_\ell(T)
    \label{eq:extpsicirccoeff}
  \end{equation}
  for $\mon_T = m_1\cdots m_{n-1}$ with $m_i = \bbb$ if $i \in T$ and $m_i = \aaa$ if $i \notin T$.
  Also the argument for \Cref{prop:topPIEStat} remains valid when considering $\mon_T = m_1\cdots m_{n-1}$ instead of $m_0\cdots m_{n-1}$.
  We get that $\#\topPIE^\circ_\ell(T)$ equals the number of pairs $(\cM,E)$ of a maximal chain~$\cM$ and a subset $E \subseteq \{1,\dots,n\}$ of cardinality $\#E = \ell$ such that $\iota(\StatMon(\cM,E)) = \mon_T$.
  This implies \Cref{eq:iotaextpsi}, finishing the proof of \Cref{cor:enab0}.
\end{proof}

We finish this section with a proof of \Cref{thm:lower-bound-on-coefficients}.
That is, we show that
\begin{equation}\label{eqn:LB}
  [t^k]\ \Num(P;1,t) \geq \binom{n-1}{k}\cdot \POIN{P}{1}\,.
\end{equation}
By \Cref{thm:maglione-voll-conjecture} and \Cref{eq:poinfromnum}, we have
\[
  \Num(P;1,t) = \extPsi^\circ(P;1,1,t)\quad\text{and}\quad
  \Poin(P;1)  = \Num(P;1,0) = \extPsi^\circ(P;1,1,0)
\]
for $\extPsi^\circ(P;y,\aaa,\bbb) = \iota\big(\extPsi(P;y,\aaa,\bbb)\big)$ as given in \Cref{eq:iotaextpsi}.
Together with \Cref{eq:extpsicirccoeff}, \Cref{eqn:LB} is equivalent to 
\begin{equation}
  \sum_T \sum_{\ell\geq 0} \#\topPIE^{\circ}_\ell(T)
  \geq
  \binom{n-1}{k}\cdot\sum_{\ell \geq 0} \#\topPIE^{\circ}_\ell(\emptyset)\,,
  \label{eq:boundingsets}
\end{equation}
where the leftmost sum runs over all subsets $T \subseteq \{1,\dots,n-1\}$ of cardinality~$k$.

\begin{lemma}\label{lem:injective}
  Let $T \subseteq \{1,\dots,n-1\}$. Then 
  \[ 
    \sum_{\ell\geq 0} \#\topPIE^{\circ}_\ell(T) \geq \sum_{\ell\geq 0}\# \topPIE^{\circ}_\ell(\emptyset).
  \]
\end{lemma}

\begin{proof}
  It is clearly sufficient to construct an injection
  \[
    \bigcup_{\ell\geq0}\topPIE^{\circ}_\ell(\emptyset) \hookrightarrow
      \bigcup_{\ell\geq0}\topPIE^{\circ}_\ell(T)\,.
  \]
  Fix $\ell\geq0$.
  Then $(\cC,\cD,\cM) \in \topPIE^{\circ}_\ell(\emptyset)$ if and only if $\cC = \{\cM_0\}$, $\cD = \multichain{\cM_\ell}$, and~$\cM$ is decreasing along
  the interval $[\cM_0,\cM_{\ell}]$ and weakly increasing along the interval
  $[\cM_{\ell},\cM_n]$.
  Now map $(\cC,\cD,\cM)$ to $(\cC_{\cM,T},\cD_{\cM,T,\ell},\cM)$.
  Here, $\cC_{\cM,T} = \{ \cM_0 \} \cup \{ \cM_i \mid i \in T \}$,
  \[
    \cD'_{\cM,T,\ell} = \multichain{ \cM_\ell} \sqcup \multichain{ \cM_{i-1} \mid i \in T, i < \ell} \sqcup \multichain{ \cM_{i+1} \mid i \in T, i > \ell}
  \]
  and $\cD_{\cM,T,\ell} = \cD'_{\cM,T,\ell}$ if $\ell \notin T$.
  If $\ell \in T$,
  \[
    \cD_{\cM,T,\ell} = \begin{cases}
            \cD'_{\cM,T,\ell} \sqcup \multichain{\cM_{\ell-1}} &\text{if } \lambda(\cM_{\ell-1},\cM_\ell) > \lambda(\cM_{\ell},\cM_{\ell+1})\,, \\
            \cD'_{\cM,T,\ell} \sqcup \multichain{\cM_{\ell+1}} &\text{if }\lambda(\cM_{\ell-1},\cM_\ell) \leq \lambda(\cM_{\ell},\cM_{\ell+1})\,. \\
          \end{cases}
  \]
  By construction, we have that $\Rank(\cC_{\cM, T}) = \{0\} \cup T$ and moreover that
  \[
    (\cC_{\cM,T},\cD_{\cM,T,\ell},\cM) \in \topPIE^{\circ}_{\ell'}(T)
  \]
  for some parameter~$\ell'$.
  This shows that the map is well-defined.
  It remains to show it is also injective.
  First note that if~$\cM$ is a maximal chain which first decreases and then only weakly increases, then there are exactly two interlacing pairs $(\cC_1,\cD_1),(\cC_2,\cD_2)$ and exactly two values of $\ell_1,\ell_2\geq 0$ such that $(\cC_i,\cD_i,\cM) \in \topPIE^{\circ}_{\ell_i}(\emptyset)$.
  Moreover,~$\ell_1$ and~$\ell_2$ differ by $1$ meaning that $\{\ell_1,\ell_2\} = \{\ell,\ell+1\}$ for some $\ell\geq 0$ and we have
  \[
    \big(\{\cM_0\},\multichain{\cM_\ell},\cM\big) \in \topPIE^{\circ}_\ell(\emptyset),\quad \big(\{\cM_0\},\multichain{\cM_{\ell+1}},\cM\big) \in \topPIE^{\circ}_{\ell+1}(\emptyset)\,.
  \]
  Injectivity follows from $\cD_{\cM,T,\ell} \neq \cD_{\cM,T,\ell+1}$ which is immediate from its definition.
\end{proof}

\begin{proof}[Proof of \Cref{thm:lower-bound-on-coefficients}]
  This follows from \Cref{eq:boundingsets} and \Cref{lem:injective}.
\end{proof}

\section{Extending the $\aaa\bbb$-index}
\label{sec:refinement-of-ber}

In this section we prove~\Cref{thm:refinement-of-ber,,cor:peakenumeration}.
The proof comes from first proving a special case of~\Cref{thm:refinement-of-ber} and then showing how to extend this to the general setting.

\medskip
As before, we fix an $R$-labeled poset~$P$ of rank~$n$ together with an $R$-labeling $\lambda$.
Recall that $\omega(\mon)$ is obtained from the monomial~$\mon$ in $\aaa,\bbb$ by substituting all occurrences of $\aaa\bbb$ with $\aaa\bbb + y\bbb\aaa + y\aaa\bbb + y^2\bbb\aaa$ and then simultaneously substituting all remaining occurrences of $\aaa$ with $\aaa+y\bbb$ and all occurrences of $\bbb$ with $\bbb+y\aaa$.
The main ingredient in the proof of the corollary is the following proposition.
Also recall the definition of $\StatMon(\cM) = \StatMon(\cM,\emptyset)$ given in \Cref{eq:statmon}.

\begin{proposition}
\label{eq:extending-equality}
  Let~$\cM$ be a maximal chain in~$P$.
  Then
  \[
    \omega\big(\StatMon(\cM)\big) = \sum_{E\subseteq\{1,\dots,n\}} y^{\#E}\cdot \StatMon(\cM,E)\,.
  \]
\end{proposition}

Before proving the proposition, we show how to deduce the first of the two corollaries.

\begin{proof}[Proof of \Cref{thm:refinement-of-ber}]
  \Cref{thm:combinatorial-interp} gives that
  \[
    \extPsi(P;y,\aaa,\bbb) = \sum_{\cM,E} y^{\#E}\cdot\StatMon(\cM,E)\,,
  \]
  and in particular,
  \[
    \Psi(P;\aaa,\bbb) = \extPsi(P;0,\aaa,\bbb) = \sum_{\cM} \StatMon(\cM)\,.
  \]
  By \Cref{eq:extending-equality}, we have 
  \[
    \extPsi(P;y,\aaa,\bbb) = \sum_{\cM} \omega\big(\StatMon(\cM)\big) = \omega\left(\sum_{\cM} \StatMon(\cM)\right) = \omega\big(\Psi(P;\aaa,\bbb)\big)\,. \qedhere
  \]
\end{proof}

For a maximal chain $\cM$, we have that $\StatMon(\cM) = m_1\dots m_{n}$ starts with $m_1 = \aaa$, compare \Cref{eq:firstindex}.
For this reason, $\StatMon(\cM)$ can be decomposed into a product of monomials of the form $\aaa \bbb^j$ for some $j \geq 0$.
This decomposition of $\StatMon(\cM)$ induces a unique decomposition of $\cM$ into intervals $[X_i,X_{i+1}]$ such that the restriction $\cM_{[X_i,X_{i+1}]}$ of $\cM$ to the interval $[X_i,X_{i+1}]$ satisfies
\[
  \StatMon(\cM_{[X_i,X_{i+1}]}) = \aaa\bbb^j
\]
for some $j\geq 0$.
We call $\{\hat{0} = X_1 < \cdots < X_k = \hat{1}\}$ the \Dfn{decomposing chain} of~$\cM$.
The following lemma shows how we can use decomposing chains to reduce the proof of~\Cref{eq:extending-equality} to a special case.

\begin{lemma}
\label{lem:dijon-subsets}
  For a maximal chain~$\cM$ with decomposing chain $\{\hat{0} = X_1 < \cdots < X_k = \hat{1}\}$ as above, we have
  \[
    \omega\big(\StatMon(\cM)\big) = \prod_{i = 1}^{k - 1} \omega\big(\StatMon(\cM_{[X_i,X_{i+1}]})\big).
  \]
\end{lemma}

\begin{proof}
  This follows from the definition of the substitution~$\omega$ and the fact that each factor~$\aaa\bbb$ in $\StatMon(\cM)$ appears in one of the intervals of the decomposition.
\end{proof}

We are now ready to prove~\Cref{eq:extending-equality} on each of the intervals given by the decomposing chain.
From there, we then use \Cref{lem:dijon-subsets} to recover $\omega\big(\StatMon(\cM)\big)$.

\begin{lemma}\label{lem:omeggs-over-easy}
  Suppose $\StatMon(\cM) = \aaa \bbb^j$ for some $j\geq 0$.
  Then
  \[
    \omega\big(\StatMon(\cM)\big) = \sum_{E\subseteq\{1,\dots,j+1\}} y^{\#E}\cdot \StatMon(\cM,E)\,.
  \]
\end{lemma}

\begin{proof}
  For $j = 0$, we have $\omega\big(\StatMon(\cM)\big) = \aaa + y\bbb$, as desired.
  For $j\geq 1$, we have
  \begin{equation}
  \label{eqn:cracking-omeggs}
    \omega\big(\StatMon(\cM)\big) = (\aaa\bbb + y\bbb\aaa + y\aaa\bbb + y^2\bbb\aaa)(\bbb + y\aaa)^{j-1}.
  \end{equation}
  There are $4\cdot 2^{j-1} = 2^{j+1}$ terms in the expansion.
  For each $E \subseteq \{1,\dots, j+1\}$, we explicitly describe a unique summand $\mathsf{s}_E$ of the expansion of the right side of \Cref{eqn:cracking-omeggs}.
  Fix $E\subseteq \{1,\dots, j+1\}$ and write $\mathsf{s}_E = s_1\cdots s_j$.
  For $i\in \{2,\dots, j\}$, we set $s_i$ to $y\bbb$ if $i\in E$ and otherwise
  to $\aaa$.
  Lastly, we set 
  \[ 
    s_1 = \begin{cases} 
      \aaa\bbb & \text{if }  1\notin E,\ j+1 \notin E, \\
      y\bbb\aaa & \text{if } 1\in E,\ j+1 \notin E, \\
      y\aaa\bbb & \text{if } 1\notin E,\ j+1 \in E, \\
      y^2\bbb\aaa & \text{if } 1 \in E,\ j+1 \in E. 
    \end{cases} 
  \] 

  It is clear that $\mathsf{s}_E$ is a summand in (the expansion of)
  \Cref{eqn:cracking-omeggs} since~$E$ encodes exactly how to select terms in the~$j$ different factors.
  Moreover, $\mathsf{s}_E = y^{\# E} \cdot \StatMon(\cM, E)$.
\end{proof}

\begin{proof}[Proof of~\Cref{eq:extending-equality}]
  Let $\{\hat{0} = X_1 < \cdots < X_k = \hat{1}\}$ be the decomposing chain for~$\cM$.
  For $i\in \{1,\dots, k-1\}$, set $\mathsf{IR}_i = \{\rank(X_i)+1,\dots,\rank(X_{i+1})\}$.
  If $E_i\subseteq \mathsf{IR}_i$ for each $i\in \{1,\dots, k-1\}$ and $E=E_1\cup \cdots \cup E_{k-1}$, then 
  \[
    \StatMon(\cM,E) = \StatMon(\cM_{[X_1,X_{2}]},E_1)\cdots\StatMon(\cM_{[X_{k-1},X_{k}]},E_{j-1}).
  \]
  Together with \Cref{lem:omeggs-over-easy,,lem:dijon-subsets}, this gives
  \begin{align*}
    \omega\big(\StatMon(\cM)\big) & = \prod_{i = 1}^{j - 1} \omega\big(\StatMon(\cM_{[X_i,X_{i+1}]})\big)\\
    & = \prod_{i = 1}^{j - 1} \left(\sum_{E_i\subseteq\mathsf{IR}_i} y^{\#E_i}\cdot \StatMon\big(\cM_{[X_i,X_{i+1}]},E_i\big) \right)\\
    & = \sum_{E\subseteq\{1,\dots,n\}} y^{\#E}\cdot \StatMon(\cM,E)\,. \qedhere
  \end{align*}
\end{proof}

The proof of the \Cref{cor:peakenumeration} uses a similar product decomposition, but is less involved.
The crucial step is to relate the substitution~$\omega$ with the \Dfn{peak enumerator}
\begin{equation}
\label{eq:thetadef}
  \theta(\cM;y) = \sum_{S} (1+y)^{\# S}\cdot y^{\bout(\cM,S)} \cdot \wt_S
\end{equation}
where the sum ranges over all $\cM$-peak-covering subsets $S \subseteq \{1,\dots,n\}$, where $\wt_S = w_1\dots w_n$ is given in \Cref{eqn:ext-weights}, and where $\bout(\cM,S)$ is the number of positions $i \notin S$ for which $u_i = \bbb$ in $\StatMon(\cM) = u_1\dots u_n$.
In particular, \Cref{cor:peakenumeration} is an immediate consequence of the following proposition.

\begin{proposition}
\label{eq:peakenumeration}
  Let $\cM$ be a maximal chain in~$P$.
  Then
  \[
    \omega\big(\StatMon(\cM)\big) = \theta(\cM;y)\,.
  \]
\end{proposition}

\begin{proof}
  First we check the proposed equality for the three special cases $\StatMon(\cM) \in \{\aaa, \bbb, \aaa\bbb\}$.
\begin{itemize}
  \item[] If $\StatMon(\cM) = \aaa$, we have $\omega(\aaa) = \aaa + y\bbb = \underbrace{(\aaa-\bbb)}_{S = \emptyset} + \underbrace{(1+y) \cdot \bbb}_{S = \{1\}}$.
  \item[] If $\StatMon(\cM) = \bbb$, we have $\omega(\aaa) = y\aaa + \bbb = \underbrace{y(\aaa-\bbb)}_{S = \emptyset} + \underbrace{(1+y) \cdot \bbb}_{S = \{1\}}$.
  \item[] If $\StatMon(\cM) = \aaa\bbb$, we have
  \begin{align*}
    \omega(\aaa\bbb)
    &= (1+y)\aaa\bbb + (y+y^2)\bbb\aaa \\
    &= \underbrace{(1+y) \cdot y \cdot \bbb(\aaa-\bbb)}_{S = \{1\}} + \underbrace{(1+y) \cdot (\aaa-\bbb)\bbb}_{S = \{2\}} + \underbrace{(1+y)^2 \cdot\bbb^2}_{S = \{1,2\}}\,.
  \end{align*}
\end{itemize}
A straightforward computation shows that if $\mon_1$ is a monomial of degree $n$ and $\mon_2$ is any monomial satisfying
\[ \peak(\mon_1) \cup \{n + i \mid \peak(\mon_2)\} = \peak(\mon_1 \mon_2),\]
then $\theta(\mon_1)\thinspace \theta(\mon_2) = \theta(\mon_1\mon_2)$ and $\omega(\mon_1)\thinspace \omega(\mon_2) = \omega(\mon_1\mon_2)$.
By decomposing $\StatMon(\cM)$ into a product of monomials of the special forms listed above (taking care that all successive $\aaa\bbb$ pairs stay in the same monomial), we obtain the result.
\end{proof}

\appendix
\section{Proof of \Cref{conj:schurpos} (by Ricky Ini Liu)}\label{sec:appendix}

In this section, we give a proof of \Cref{conj:schurpos} by relating the map $\omega$ to the internal (Kronecker) coproduct on quasisymmetric and symmetric functions.

\subsection{Internal coproduct}
Given two (ordered) sets of variables $\mathbf x = \{x_1, x_2, \dots\}$ and $\mathbf y = \{y_1, y_2, \dots\}$, denote by $\mathbf{xy} = \{x_1y_1, x_1y_2, \dots, x_2y_1, x_2y_2, \dots\}$ the set of pairwise products $x_iy_j$ in lexicographic order. Then for any quasisymmetric function $f \in \QSym$, we can decompose the power series $f(\mathbf{xy}) \in \QQ[[\mathbf{xy}]] \subseteq \QQ[[\mathbf x]] \otimes \QQ[[\mathbf y]]$ into a sum of the form
\[f(\mathbf{xy}) = \sum_i g_i(\mathbf x)h_i(\mathbf y),\]
where $g_i, h_i \in \QSym$.
The \Dfn{internal coproduct} (or \Dfn{Kronecker coproduct}) of $\QSym$ is the map $\Delta\colon \QSym \to \QSym \otimes \QSym$ given by
\[\Delta(f) = \sum_i g_i \otimes h_i.\]
(See, for instance, \cite{gessel} for more discussion.)

The following result due to Gessel expresses the coproduct in terms of the fundamental basis $F_S$. (Recall that for a permutation $\pi \in S_n$, $\des(\pi) = \{i \mid \pi(i) > \pi(i+1)\}$.)

\begin{proposition}[\!\!{{\cite[Theorem 11]{gessel}}}] \label{prop:qsym-coproduct}
    Let $\pi \in S_n$. Then
    \[\Delta(F_{\des(\pi)}) = \sum_{\tau \sigma = \pi} F_{\des(\sigma)} \otimes F_{\des(\tau)},\]
    where the sum ranges over all pairs of permutations $\sigma,\tau \in S_n$ such that $\tau\sigma = \pi$.
\end{proposition}

The internal coproduct on $\QSym$ restricts to a coproduct on symmetric functions. On a Schur function $s_\lambda$, we have
\[\Delta(s_\lambda) = \sum_{\mu,\nu} g_{\lambda\mu\nu} s_\mu \otimes s_\nu,\]
where $g_{\lambda\mu\nu}$ are the \Dfn{Kronecker coefficients}. One can also define the \Dfn{internal} or \Dfn{Kronecker product} $*$ on symmetric functions in terms of Kronecker coefficients:
\[s_\lambda * s_\mu = \sum_\nu g_{\lambda\mu\nu} s_\nu.\]
It is well known that the Kronecker coefficients $g_{\lambda\mu\nu}$ (which are symmetric in $\lambda$, $\mu$, and $\nu$) are nonnegative integers since they describe the irreducible decomposition of the tensor product of two irreducible representations of $S_n$.

\subsection{Proof of \Cref{conj:schurpos}}
The key lemma is the following description of the map $\omega$ as defined on $\QSym$ in \Cref{subsec:QSYM-connection}. Let $\varphi$ be the linear map on $\QSym$ defined by
\[\varphi(F_S) = \begin{cases} y^k &\text{if $S = \{1, 2, \dots, k\}$ for some $0 \leq k \leq n-1$,}\\ 0&\text{otherwise}.\end{cases}\]
Also recall that for $1 < i < n$, we say $i$ is a \Dfn{peak} of a permutation $\pi \in S_n$ if $\pi(i-1) < \pi(i) > \pi(i+1)$, while $i$ is a \Dfn{valley} if $\pi(i-1) > \pi(i) < \pi(i+1)$.

\begin{lemma} \label{lem:omega-f}
    Let $\pi \in S_n$. Then $\omega = (\mathrm{id} \otimes \varphi) \circ \Delta$, so that
    \[\omega(F_{\des(\pi)}) = \sum_{\tau\sigma = \pi} F_{\des(\sigma)} \cdot \varphi(F_{\des(\tau)}).\]
\end{lemma}
\begin{proof}
    Suppose $\des(\tau) = \{1, \dots, k\}$ for some $0 \leq k \leq n-1$, and let $\sigma = \tau^{-1} \pi$. Each such $\tau$ is determined by the subset $\{\tau(1), \dots, \tau(k)\} \subseteq [n] \setminus \{1\}$, and so the set $E = \{\sigma^{-1}(1), \dots, \sigma^{-1}(k)\}$ can be any of the $2^{n-1}$ subsets of $[n] \setminus \{\pi^{-1}(1)\}$. Then we can describe the transformation \[\omega' \colon \m_{\des(\pi)} \mapsto \sum_{\substack{\tau\sigma = \pi\\\des(\tau) = \{1, \dots, k\}}} y^k \m_{\des(\sigma)}\]
    by considering the effect when each $j \neq \pi^{-1}(1)$ is chosen to be in $E$ or not.
    
    Note that if $i < j$, then $\tau^{-1}(i) > \tau^{-1}(j)$ if and only if $j \in \{\tau(1), \dots, \tau(k)\}$. Then for $\pi(i) < \pi(j)$, we have $\sigma(i) > \sigma(j)$ if and only if $j \in E$.
    \begin{itemize}
        \item Suppose $j$ is a peak of $\pi$. If $j \notin E$, then $j$ is a peak of $\sigma$, while if $j \in E$, then $j$ is a valley of $\sigma$. Thus $\omega'$ replaces each $\mathbf{ab}$ with $\mathbf{ab} + y \mathbf{ba}$.
        \item Suppose $j$ is not a peak of $\pi$ but it is a descent. If $j \notin E$, then $j$ is a descent of $\sigma$, but if $j \in E$, then $j$ is an ascent of $\sigma$. Thus $\omega'$ replaces each remaining $\mathbf{b}$ with $\mathbf b + y\mathbf a$.
        \item Suppose $j$ is not a peak of $\pi$ but $j-1$ is an ascent of $\pi$. If $j \notin E$, then $j-1$ is an ascent of $\sigma$, but if $j \in E$, $j-1$ is a descent of $\sigma$. Thus $\omega'$ replaces each remaining $\mathbf{a}$ with $\mathbf a + y\mathbf b$.
        \item Otherwise, $\pi(j)$ is smaller than both $\pi(j-1)$ and $\pi(j+1)$ (if they exist), and the descents of $\pi$ are unaffected by whether $j \in E$ or not. Thus $\omega'$ multiplies the result by $1+y$ for each $j \neq \pi^{-1}(1)$ that falls in this case. Counting $j = \pi^{-1}(1)$, exactly one such $j$ occurs before, after, and between each peak. Therefore the number of extra factors of $1+y$ equals the number of peaks of $\pi$.
    \end{itemize}
    Comparing to the definition of $\omega$ in \Cref{thm:refinement-of-ber} shows that $\omega' = \omega$ on $\mathbf{ab}$-monomials, from which the result easily follows by \Cref{prop:qsym-coproduct}.
\end{proof}

We are now able to prove the following theorem, which resolves \Cref{conj:schurpos}.
\begin{thm}\label{thm:proof-of-conjecture}
    For any partition $\lambda \vdash n$,
    \[\omega(s_\lambda) = \sum_{k=0}^{n-1} (s_\lambda * s_{(n-k,1^k)}) y^k = \sum_{\mu \vdash n} c^\mu_{\lambda}(y) s_\mu,\] where \[c^\mu_{\lambda}(y) = \sum_{k=0}^{n-1} g_{\lambda,\mu,(n-k,1^k)} y^k \in \NN[y].\]
\end{thm}
\begin{proof}
    When writing the Schur function $s_\nu$ in the fundamental basis, the coefficient of $F_{\{1, 2, \dots, k\}}$ is $1$ if $\nu = (n-k, 1^k)$, and $0$ otherwise. (This is because the unique standard Young tableau of size $n$ with descent set $\{1, 2, \dots, k\}$ has shape $(n-k, 1^k)$, or see \cite[Theorem 3]{gessel}.) It follows from \Cref{lem:omega-f} that
    \begin{align*}
        \omega(s_\lambda) &= (\mathrm{id} \otimes \varphi) \circ \Delta(s_\lambda)\\
        &= \sum_{\mu, \nu \vdash n} g_{\lambda\mu\nu} s_\mu \cdot \varphi(s_\nu)\\
        &= \sum_{\mu \vdash n} \sum_{k=0}^{n-1} g_{\lambda,\mu,(n-k,1^k)}y^k \cdot  s_\mu\\
        &= \sum_{k=0}^{n-1} (s_\lambda * s_{(n-k,1^k)}) y^k.\qedhere
    \end{align*}
\end{proof}

\bibliographystyle{plain}
\bibliography{bibliog}

\begin{thebibliography}{10}

\bibitem{bayer-survey}
Margaret~M. Bayer.
\newblock The {$\ccc\ddd$}-index: a survey.
\newblock In {\em Polytopes and discrete geometry}, volume 764 of {\em Contemp. Math.}, pages 1--19. Amer. Math. Soc., 2021.

\bibitem{bayer-billera}
Margaret~M. Bayer and Louis~J. Billera.
\newblock Generalized {D}ehn-{S}ommerville relations for polytopes, spheres and {E}ulerian partially ordered sets.
\newblock {\em Invent. Math.}, 79(1):143--157, 1985.

\bibitem{bayer-klapper}
Margaret~M. Bayer and Andrew Klapper.
\newblock A new index for polytopes.
\newblock {\em Discrete Comput. Geom.}, 6(1):33--47, 1991.

\bibitem{bergeron-mykytiuk-sottile-vanWilligenburg2000}
Nantel Bergeron, Stefan Mykytiuk, Frank Sottile, and Stephanie van Willigenburg.
\newblock Noncommutative {P}ieri operators on posets.
\newblock {\em J. Combin. Theory Ser. A}, 91(1-2):84--110, 2000.

\bibitem{billera-ehrenborg-readdy}
Louis~J. Billera, Richard Ehrenborg, and Margaret Readdy.
\newblock The {$\ccc$-$2\ddd$}-index of oriented matroids.
\newblock {\em J. Combin. Theory Ser. A}, 80(1):79--105, 1997.

\bibitem{bjorner-shellings}
Anders Bj\"{o}rner.
\newblock Shellable and {C}ohen-{M}acaulay partially ordered sets.
\newblock {\em Trans. Amer. Math. Soc.}, 260(1):159--183, 1980.

\bibitem{bjorner-garsia-stanley}
Anders Bj{\"o}rner, Adriano~M. Garsia, and Richard~P. Stanley.
\newblock An introduction to {C}ohen-{M}acaulay partially ordered sets.
\newblock In {\em Ordered Sets}, pages 583--615. Springer Netherlands, 1982.

\bibitem{BjornerEtAl}
Anders Bj\"{o}rner, Michel Las~Vergnas, Bernd Sturmfels, Neil White, and G\"{u}nter~M. Ziegler.
\newblock {\em Oriented {M}atroids}.
\newblock Encyclopedia of Mathematics and its Applications, Volume 46. Cambridge University Press, second edition, 1999.

\bibitem{bland-lasvergnas}
Robert~G. Bland and Michel Las~Vergnas.
\newblock Orientability of {M}atroids.
\newblock {\em J. Combinatorial Theory Ser. B}, 24(1):94--123, 1978.

\bibitem{Budur-Nero-Mustata}
Nero Budur, Mircea Musta\c{t}\u{a}, and Zach Teitler.
\newblock The monodromy conjecture for hyperplane arrangements.
\newblock {\em Geom. Dedicata}, 153:131--137, 2011.

\bibitem{Denef}
Jan Denef.
\newblock Report on {I}gusa's local zeta function.
\newblock Number 201-203, pages Exp. No. 741, 359--386. 1991.
\newblock S\'{e}minaire Bourbaki, Vol. 1990/91.

\bibitem{duSautoy-Grunewald}
Marcus du~Sautoy and Fritz~J. Grunewald.
\newblock Analytic properties of zeta functions and subgroup growth.
\newblock {\em Ann. of Math. (2)}, 152(3):793--833, 2000.

\bibitem{Ehrenborg-signed}
Richard Ehrenborg.
\newblock The {$r$}-signed {B}irkhoff transform.
\newblock {\em Discrete Math.}, 344(2), 2021.

\bibitem{Ehrenborg-Readdy-transforms}
Richard Ehrenborg and Margaret Readdy.
\newblock The {T}chebyshev transforms of the first and second kind.
\newblock {\em Ann. Comb.}, 14(2):211--244, 2010.

\bibitem{gessel}
Ira~M. Gessel.
\newblock Multipartite {$P$}-partitions and {I}nner {P}roducts of {S}kew {S}chur functions.
\newblock 34:289--317, 1984.

\bibitem{2023arXiv230100309G}
Darij {Grinberg} and Ekaterina~A. {Vassilieva}.
\newblock {The algebra of extended peaks}.
\newblock {\em arXiv e-prints}, page arXiv:2301.00309, December 2022.

\bibitem{Grunewald-Segal-Smith}
Fritz~J. Grunewald, Daniel~M. Segal, and Geoff~C. Smith.
\newblock Subgroups of finite index in nilpotent groups.
\newblock {\em Invent. Math.}, 93(1):185--223, 1988.

\bibitem{Hsiao}
Samuel~K. Hsiao.
\newblock A signed analog of the {B}irkhoff transform.
\newblock {\em J. Combin. Theory Ser. A}, 113(2):251--272, 2006.

\bibitem{Jensen-Kutler-Usatine}
David Jensen, Max Kutler, and Jeremy Usatine.
\newblock The motivic zeta functions of a matroid.
\newblock {\em J. Lond. Math. Soc. (2)}, 103(2):604--632, 2021.

\bibitem{karu}
Kalle Karu.
\newblock The cd-index of fans and posets.
\newblock {\em Compositio Mathematica}, 142(3):701--718, 2006.

\bibitem{Kuhne-Maglione}
Lukas K\"uhne and Joshua Maglione.
\newblock On the geometry of flag {Hilbert{\textendash}Poincar\'e} series for matroids.
\newblock {\em Algebraic Combinatorics}, 6(3):623--638, 2023.

\bibitem{Maglione-Voll}
Joshua Maglione and Christopher Voll.
\newblock Flag {H}ilbert--{P}oincar{\'e} series of hyperplane arrangements and {I}gusa zeta functions.
\newblock {\em Israel Journal of Mathematics (to appear)}, 2023.

\bibitem{Meisinger}
G{\"{u}}nter Meisinger.
\newblock {\em Flag numbers and quotients of convex polytopes}.
\newblock PhD thesis, University of Passau, Germany, 1994.

\bibitem{Rossmann}
Tobias Rossmann.
\newblock Computing topological zeta functions of groups, algebras, and modules, {I}.
\newblock {\em Proc. Lond. Math. Soc. (3)}, 110(5):1099--1134, 2015.

\bibitem{Rossmann-Voll}
Tobias Rossmann and Christopher Voll.
\newblock Groups, graphs, and hypergraphs: average sizes of kernels of generic matrices with support constraints, 2019.
\newblock to appear in Mem. Amer. Math. Soc.

\bibitem{saliola-thomas}
Franco Saliola and Hugh Thomas.
\newblock Oriented interval greedoids.
\newblock {\em Discrete \& Computational Geometry}, 47(1):64--105, 2012.

\bibitem{stembridge-p-partitions}
John~R. Stembridge.
\newblock Enriched {$P$}-partitions.
\newblock {\em Trans. Amer. Math. Soc.}, 349(2):763--788, 1997.

\bibitem{wachs-walker}
Michelle~L. Wachs and James~W. Walker.
\newblock On geometric semilattices.
\newblock {\em Order}, 2(4):367--385, 1986.

\bibitem{ziegler-nonorientable}
G\"{u}nter~M. Ziegler.
\newblock Some minimal nonorientable matroids of rank three.
\newblock {\em Geom. Dedicata}, 38(3):365--371, 1991.

\bibitem{ziegler-shellability}
G\"{u}nter~M. Ziegler.
\newblock Matroid shellability, {$\beta$}-systems, and affine hyperplane arrangements.
\newblock {\em J. Algebraic Combin.}, 1(3):283--300, 1992.

\bibitem{ziegler}
G\"{u}nter~M. Ziegler.
\newblock {\em Lectures on {P}olytopes}.
\newblock Graduate Texts in Mathematics, Volume 152. Springer-Verlag, New York, 1995.

\end{thebibliography}

\end{document}